\documentclass[letterpaper,10pt]{article}

\usepackage{latexsym}
\usepackage{amscd}
\usepackage{graphicx}
\usepackage{amsmath}
\usepackage{amssymb}
\usepackage{mathrsfs}
\usepackage{amsthm}
\usepackage{bbm}
\usepackage{color}
\usepackage{accents}
\usepackage{enumerate}
\usepackage{epsfig}
\usepackage{psfrag}
\usepackage{hyperref}
\input xy
\xyoption{all}

\newcommand{\N}{\mathbb{N}}                     
\newcommand{\Z}{\mathbb{Z}}                     
\newcommand{\R}{\mathbb{R}}                     
\newcommand{\C}{\mathbb{C}}                     
\newcommand{\D}{\mathbb{D}}                     
\newcommand{\im}{\mathrm{Im\,}}                 
\newcommand{\re}{\mathrm{Re\,}}                 
\newcommand{\supp}{\mathrm{supp\,}}             

\newcommand{\Cal}{\mathrm{CAL}}		
\newcommand{\Fix}{\mathrm{Fix}}		

\newtheorem{thm}{\sc Theorem}[section]          
\newtheorem*{thm*}{\sc Theorem}               	
\newtheorem*{cor*}{\sc Corollary}        		
\newtheorem{lem}[thm]{\sc Lemma}            
\newtheorem{prop}[thm]{\sc Proposition}     
\newtheorem{add}[thm]{\sc Addendum}         

\title{Systolic ratio, index of closed orbits and convexity for tight contact forms on the three-sphere}

\author{Alberto Abbondandolo, Barney Bramham, \\ Umberto L.~Hryniewicz, Pedro A.~S.~Salom\~ao}

\date{}

\begin{document}

\maketitle

\begin{abstract} 
We construct a dynamically convex contact form on the three-sphere whose systolic ratio is arbitrarily close to 2. This example is related to a conjecture of Viterbo, whose validity would imply that  the systolic ratio of a convex contact form does not exceed 1. We also construct a sequence of tight contact forms $\alpha_n$, $n\geq 2$, with systolic ratio arbitrarily close to $n$ and suitable bounds on the mean rotation number of all the closed orbits of the induced Reeb flow.
\end{abstract}



\section{Introduction}

\subsection{Dynamically convex contact forms on the three-sphere}

Let $\alpha$ be a contact form on the three-sphere $S^3$, i.e.\ a smooth 1-form such that $\alpha\wedge d\alpha$ is a volume form. The kernel of $\alpha$ is called the contact structure induced by $\alpha$.
We denote by $R_\alpha$  the Reeb vector field of $\alpha$, i.e.\ the unique vector field on $S^3$ determined by the conditions $d\alpha(R_\alpha, \cdot) = 0$ and $\alpha(R_\alpha)= 1$. The flow of this vector field is called the Reeb flow of $\alpha$. 

A closed Reeb orbit of $\alpha$ is a pair $(\gamma,T)$, where $\gamma:\R \to S^3$ is a periodic trajectory of the Reeb flow of $\alpha$  and $T>0$ is a period of $\gamma$ (not necessarily the minimal one).  The set of all closed Reeb orbits of $\alpha$ is denoted by $\mathcal{P}(\alpha)$. The action of $(\gamma,T)$, i.e.\ the quantity 
\[
\mathcal{A}(\gamma,T):=\int_{\R/T\Z} \gamma^*\alpha=T,
\]
coincides with the period $T$, since $\alpha(\dot \gamma)= \alpha(R_{\alpha}\circ \gamma)= 1$. The minimal period, or minimal action, of $\alpha$ is the positive number 
\[
T_{\rm min}(\alpha) := \min_{(\gamma,T) \in \mathcal{P}(\alpha)} \mathcal{A}(\gamma,T)=\min_{(\gamma,T) \in \mathcal{P}(\alpha)}T.
\]
The fact that Reeb flows on $S^3$ always have closed orbits (see \cite{rab78} and \cite{hof93}, or \cite{tau07} for the case of a general closed three-manifold) ensures that $T_{\min}(\alpha)$ is well-defined.
This quantity scales linearly when $\alpha$ is multiplied by a constant: $T_{\min}(c\alpha) = |c| T_{\min}(\alpha)$ for all $c\in \R\setminus \{0\}$. The scale invariant quantity
\[
\rho_{\mathrm{sys}}(\alpha) := \frac{T_{\min}(\alpha)^2}{\mathrm{vol}(S^3,\alpha\wedge d\alpha)}
\]
is called the systolic ratio of $\alpha$. It is a natural generalization of a corresponding quantitative invariant in Riemannian and Finlser geometry, see \cite{apb14,abhs17b}.

Two contact forms $\alpha$ and $\beta$ are said to be equivalent if $\beta=\varphi^* \alpha$ for some diffeomorphism $\varphi$. In this case, the contact volumes are the same and the Reeb flows are conjugated by $\varphi$. In particular, $\alpha$ and $\beta$ have the same systolic ratio. 

In this paper, we are interested in tight contact forms on $S^3$, which can be characterized as those contact forms which are equivalent to contact forms inducing the standard contact structure $\ker \alpha_0$, where the contact form $\alpha_0$ is the restriction to $S^3=\{z\in \C^2 \mid |z|=1\}$ of the Liouville form
\[
\Lambda:=\frac{1}{2} \sum_{j=1}^2 (x_j \, dy_j - y_j \, dx_j),
\]
where $(x_1+iy_,x_2+iy_2)$ are the standard coordinates in $\C^2$. In other words, tight contact forms on $S^3$ are, up to equivalence, of the form $\alpha=f\alpha_0$, where $f$ is a non-vanishing smooth real function on $S^3$. Since $T_{\min}$ and the contact volume do not change when $\alpha$ is replaced by $-\alpha$, we may assume that the function $f$ is positive. 

Tight contact forms on $S^3$ are precisely the contact forms which are induced by starshaped hypersurfaces in $\C^2$. More precisely, we consider a compact domain $A\subset \C^2$ whose interior contains the origin and whose boundary $\partial A$ is smooth and transverse to the radial direction. Then $\Lambda$ restricts to a contact form on $\partial A$, and the diffeomorphism $S^3 \rightarrow \partial A$ which is given by the radial projection pulls this contact form back to a tight contact form $\alpha_A$ on $S^3$. More precisely,
\[
\alpha_A = a^2 \alpha_0,
\]
where $a:S^3 \rightarrow (0,+\infty)$ is the smooth function such that
\[
A = \{ rz \mid z\in S^3, \; 0\leq r \leq a(z)\}.
\]
The above formula shows that all tight contact forms are obtained in this way.
If two starshaped domains are symplectomorphic, i.e.\ diffeomorphic by a diffeomorphism which preserves the standard symplectic form 
\[
d\Lambda = dx_1\wedge dy_1 + dx_2\wedge dy_2,
\]
then the corresponding contact forms on $S^3$ are equivalent. 

When the domain $A$ happens to be convex, we say that the contact form $\alpha_A$ is convex. Convexity is not preserved by symplectomorphisms of $\C^2$, but we can saturate this definition by declaring to be convex any contact form on $S^3$ which is equivalent to one of the form $\alpha_A$, with $A$ convex.

Similarly, a contact form on $S^3$ is said to be strongly convex if it is equivalent to one of the form $\alpha_A$, with $A$ strongly convex, meaning that $\partial A$ has positive sectional curvature.

Convexity and strong convexity are difficult to characterize by purely symplectic or contact topological conditions. However, Hofer, Wysocki and Zehnder proved that if the contact form $\alpha$ on $S^3$ is strongly convex, then all the closed orbits $(\gamma,T)$ of the Reeb flow $\phi^t_{R_{\alpha}}$ of $\alpha$ have Conley-Zehnder index at least 3:
\[
\mu(\gamma,T)\geq 3.
\]
Here, the Conley-Zehnder index of $(\gamma,T)$ is defined to be the Conley-Zehnder index of the path of symplectic automorphisms of $\R^2$ which is obtained from
\[
d\phi^t_{R_{\alpha}}|_{\ker \alpha(\gamma(0))} : \ker \alpha(\gamma(0)) \rightarrow \ker \alpha(\gamma(t))
\]
once the contact structure $\ker \alpha$ is trivialized by a global symplectic trivialization which is isotopic to the standard global trivialization of $\ker \alpha_0$. Actually, the fact that closed orbits might be degenerate forces one to work with a suitable extension of the Conley-Zehnder index to degenerate symplectic paths. See Sections \ref{ssection_CZ_index} and \ref{Sec:fromTtoS} below for more details on the Conley-Zehnder index and on the standard global trivialization of $\ker \alpha_0$. 

Tight contact forms on $S^3$ all of whose closed Reeb orbits have Conley-Zehnder index at least 3 are called dynamically convex. Dynamical convexity is on the nose an invariant notion, meaning that if $\alpha$ is dynamically convex, so is any contact form which is equivalent to $\alpha$.

Strongly convex contact forms are dynamically convex, but there is at the time of writing no known example of a dynamically convex contact form which is not strongly convex. The difficulty in finding such an example is due to the fact that, besides the conditions on the Conley-Zehnder index of closed orbits, it is difficult to come up with properties of a contact form which imply that it is not equivalent to some $\alpha_A$ with $A$ convex, or strongly convex.

Convexity plays an important role in a very interesting conjecture of Viterbo. Its original formulation in \cite{vit00} states that if $c$ is a symplectic capacity on domains of $\C^n$ then
\begin{equation}
\label{vit}
c(A)^n \leq n! \, \mathrm{vol}(A)
\end{equation}
for any compact convex domain $A\subset \C^n$, with the equality holding if and only if $A$ is symplectomorphic to a ball. Here, $\mathrm{vol}$ denotes the Euclidean $2n$-dimensional volume. This conjecture is wide open, except for the trivial case $n=1$ in which any symplectic capacity must agree with the area and hence the equality holds in \eqref{vit}. In \cite{vit00}, Viterbo proved a weaker inequality, in which the coefficient $n!$ is replaced by a faster growing function of $n$. This result was improved by Artstein-Avidan, Milman and Ostrover, who proved \eqref{vit} with $n!$ replaced by $C \,n!$ for some large number $C$, see \cite{amo08}. The sharp inequality \eqref{vit} is known only for special convex domains, such as domains which are invariant under the multiplication by complex numbers of modulus 1, and, in the case $n=2$ and for the Hofer-Zehnder capacity, for all convex domains which are close enough to a ball, see \cite[Theorem 1]{abhs17b}. It is also interesting to notice that the sharp bound \eqref{vit} in the case of the Hofer-Zehnder capacity and for special domains of the form $A=K \times K^{\circ}$, where $K$ is a compact convex centrally symmetric neighborhood of the origin in $\R^n$ and $K^{\circ}\subset i\R^n$ denotes its polar, implies the $n$-dimensional Mahler conjecture in convex geometry, see \cite{ako14}.

Now let's go back to the case $n=2$. It is well known that the Hofer-Zehnder capacity of a compact convex domain $A\subset \C^2$ with smooth boundary coincides with $T_{\min}(\alpha_A)$, see \cite[Proposition 4]{hz90}. Moreover, the Euclidean 4-dimensional volume of $A$ coincides with twice its volume with respect to the volume form $d\Lambda\wedge d\Lambda$ on $\C^2$, and by Stokes theorem with twice the volume of $\partial A$ with respect to $\alpha_A\wedge d\alpha_A$. Therefore, the 4-dimensional Viterbo conjecture for the Hofer-Zehnder capacity is equivalent to the fact that
\begin{equation}
\label{vit2}
\rho_{\mathrm{sys}}(\alpha)\leq 1
\end{equation}
for every convex contact form on $S^3$, with the equality holding if and only if all Reeb orbits of $\alpha$ are closed and have the same period. The latter assertion follows from the fact that a starshaped domain $A\subset \C^2$ is symplectomorphic to a closed ball if and only if all Reeb orbits of $\alpha_A$ are closed and have the same period, see \cite[Proposition 4.3]{abhs17b}. All of this, except for the last assertion, would extend to any $n$, but here we focus on the case $n=2$. Notice also that the validity of \eqref{vit2} for strongly convex contact forms would imply it for all convex contact forms, because on the space of convex contact forms the systolic ratio is $C^0$-continuous.

Many symplectic results holding for strongly convex contact forms extend to dynamically convex ones, the reason being that modern symplectic techniques involving $J$-holomorphic curves do not distinguish between strong convexity and dynamical convexity. Therefore, one is tempted to attack the 4-dimensional Viterbo conjecture for the Hofer-Zehnder capacity by proving the bound \eqref{vit2} for dynamically convex contact forms on $S^3$. Our first result excludes this possibility:

\begin{thm}
\label{T:main_DC}
For every $\epsilon>0$ there is a dynamically convex contact form $\alpha$ on $S^3$ such that
\[
2-\epsilon < \rho_{\mathrm{sys}}(\alpha) < 2.
\]
In particular, the supremum of the systolic ratio over all dynamically convex contact forms on $S^3$ is at least 2.
\end{thm}

In particular, when $\epsilon\leq 1$ the contact form $\alpha$ of the above theorem is either a counterexample to the Viterbo conjecture or the first example of a dynamically convex contact form on $S^3$ which is not convex. Regrettably, we do not know which of these two alternatives holds. 
We also remark that is not known whether the supremum of the systolic ratio over all dynamically convex contact forms on $S^3$ is finite.

The proof of Theorem \ref{T:main_DC} is based on a construction which we introduced in \cite{abhs17b}. The argument consists in constructing a special area preserving diffeomorphism $\varphi$ of the disk $\D:=\{z\in \C:|z|\leq 1\}$ which is supported in the interior of $\D$, whose Calabi invariant is small when compared to the action of its periodic points, and such that the Conley-Zehnder index of all its fixed points is at least $-1$ and that of all its $k$-periodic points is at least $1-3k$. Such a disk map is then embedded as the first return map to a disk-like global surface of section for the Reeb flow of a contact form on $S^3$, which turns out to be tight because it is contact isotopic to $\alpha_0$, dynamically convex because of the bounds on the Conley-Zehnder indices of periodic points of $\varphi$, and whose systolic ratio belongs to the interval $(2-\epsilon,2)$ because of the bounds on the action of periodic points and on the Calabi invariant of $\varphi$. The novelty with respect to \cite[Theorem 2]{abhs17b} consists in the careful analysis which is necessary in order to get the lower bounds on the Conley-Zehnder indices of the periodic points of $\varphi$ and in the study of the relationship between the Conley-Zehnder indices of closed orbits of a Reeb flow admitting a global surface of section and the Conley-Zehnder indices of the corresponding periodic points of the first return map.

\subsection{Higher systolic ratios and the mean rotation number of closed orbits}

In \cite[Theorem 2]{abhs17b} we constructed tight contact forms on $S^3$ whose systolic ratio is arbitrarily large. These contact forms are not dynamically convex, and our next aim here is to understand more about what happens to the Conley-Zehnder indices of closed Reeb orbits when one tries to make the systolic ratio high. In our construction closed orbits with negative Conley-Zehnder index appear. When such an orbit is iterated its Conley-Zehnder index becomes more and more negative, so we do not expect the Conley-Zehnder indices of closed orbits of a Reeb flow with high systolic ratio to be bounded from below. And since the  $k$-th iteration of a closed orbit of minimal period $T$ might produce after a perturbation a closed orbit of minimal period close to $kT$ and index close to the index of the $k$-th iteration of the original orbit, we do not expect the Conley-Zehnder indices of simple closed orbits to be bounded from below either. What we can expect to have good lower bounds on, is instead the mean rotation number of closed orbits, whose definition we now recall.

Let $\alpha$ be a tight contact form on $S^3$ and let $(\gamma,T)$ be a closed orbit of the Reeb flow of $\alpha$. We define the mean rotation number of $(\gamma,T)$ as half the mean Conley-Zehnder index of $(\gamma,T)$, normalized by the period:
\[
\overline{\rho}(\gamma,T) := \frac{1}{2} \lim_{k\rightarrow +\infty} \frac{\mu(\gamma,kT)}{kT}.
\] 
The limit exists due to the quasi-morphism property of the Conley-Zehnder index. The presence of the factor $1/2$ is due to the fact that the mean Conley-Zehnder index of a full rotation is 2. One easily checks that
\[
\overline{\rho}(\gamma,kT) = \overline{\rho}(\gamma,T) \qquad \forall k\in \N,
\]
and hence we may remove the period $T$ from the notation and write
\[
\overline{\rho}(\gamma) = \overline{\rho}(\gamma,T).
\]
One can also show that for a closed orbit $(\gamma,T)$ the implication 
\[
\overline{\rho}(\gamma) < 0 \quad \Rightarrow \quad \mu(\gamma,kT) <0 \qquad \forall k\in \N
\]
holds. Now we would like to consider the infimum and supremum of the mean rotation number of all closed orbits of $R_{\alpha}$. In order to obtain quantities which are invariant under the rescaling $\alpha \mapsto c\,\alpha$ we multiply the mean rotation numbers by $T_{\min}(\alpha)$ and define
\[
\begin{aligned}
s(\alpha): =T_{\min}(\alpha) \inf_{(\gamma,T) \in \mathcal{P}(\alpha)}\overline{\rho}(\gamma),\\  S(\alpha): = T_{\min}(\alpha)\sup_{(\gamma,T) \in \mathcal{P}(\alpha)} \overline{\rho}(\gamma).
\end{aligned}
\]
Finally, we define
\[
\Delta(\alpha): = S(\alpha)-s(\alpha).
\]
By the compactness of $S^3$, the ratio $\mu(\gamma,kT)/kT$ has uniform bounds for all closed orbits $(\gamma,T)$ and hence
\[
-\infty<s(\alpha) \leq S(\alpha)<+\infty,
\]
which implies
\[
0\leq \Delta(\alpha) <+\infty.
\]

As an example, if $\alpha=\alpha_0$ is the standard contact form on $S^3$ then all of its  Reeb orbits are covers of the Hopf fibers, which all have least period $\pi$. Hence $T_{\min}(\alpha_0)=\pi$. A simple computation shows that 
\[
\overline{\rho}(\gamma)= \frac{2}{\pi}\qquad \forall (\gamma,T) \in \mathcal{P}(\alpha_0).
\]
It follows that 
\[
s(\alpha_0) =S(\alpha_0)=2 \qquad \mbox{and} \qquad \Delta(\alpha_0)=0.
\]
Notice that in this case ${\rm vol}(S^3,\alpha_0\wedge d\alpha_0)=\pi^2$ and hence $\rho_{\rm sys}(\alpha_0) =1$. 

Our next result generalizes at the same time Theorem \ref{T:main_DC} and the second part of \cite[Theorem 2]{abhs17b}. It shows that one can find tight contact forms on $S^3$ with  systolic ratio arbitrarily close to any integer $n\geq 2$ in such a way that $S$ has the value 2 and $s$ is close to the integer $-(n-1)^2+2$, which is negative for $n\geq 3$ and tends to $-\infty$ quadratically.

\begin{thm}\label{thm0}
For  each natural number $n \geq 2$ and each $\epsilon>0$ there exists  a tight contact form $\alpha$ on $S^3$  whose systolic ratio satisfies
\[
n-\epsilon< \rho_{\mathrm{sys}}(\alpha) <n,
\]
such that the invariants $s(\alpha)$ and $S(\alpha)$ satisfy
\[
-(n-1)^2+ 2 < s(\alpha) < -(n-1)^2+ 2 + \epsilon \qquad \mbox{ and }\qquad  S(\alpha)=2,
\]
and which is dynamically convex for $n=2$. In particular,
\[
(n-1)^2 - \epsilon < \Delta(\alpha) < (n-1)^2,
\]
and $\alpha$ admits Reeb orbits with negative Conley-Zehnder indices for all $n\geq 3$.
\end{thm}

The proof of the above result follows the same scheme of that of Theorem \ref{T:main_DC}, but requires a more general family of area-preserving diffeomorphisms of the disk, with precise bounds relating the mean Conley-Zehnder indices of periodic points to their action.

We do not know whether Theorems \ref{T:main_DC} and \ref{thm0} are sharp in the following sense:
\begin{enumerate}
\item[(i)] Is it true that $\rho_{\rm sys}(\alpha) < 2$ for all dynamically convex contact forms $\alpha$ on $S^3$?
\item[(ii)] Is it true that $s(\alpha)$ goes to $-\infty$ as $\rho_{\rm sys}(\alpha)$ tends to $+\infty$? 
\item[(iii)] More ambitously, is it true that if $\rho_{\rm sys}(\alpha)\geq  n$ for some $n\in \N$ then
\[
s(\alpha) \leq -(n-1)^2+2 \qquad \mbox{ and } \qquad \Delta(\alpha) \geq  (n-1)^2?
\]
\end{enumerate}
An affirmative answer to question (i) would imply that the supremum of the systolic ratio over all dynamically convex contact forms on $S^3$ is 2.
An affirmative answer to question (iii) would imply the existence of closed Reeb orbits with negative Conley-Zehnder index whenever $\rho_{\rm sys}(\alpha)\geq 3$.

The paper is organized in the following way. In Section \ref{secprima} we review the definitions and main properties of the action, Calabi invariant and Conley-Zehnder index and we construct the family of special  area-preserving diffeomorphisms of the disk. In Section \ref{secseconda} we discuss the issue of lifting area-preserving diffeomorphisms of the disk to Reeb flows on $S^3$, with special regard to the behaviour of the Conley-Zehnder index, and finally we prove Theorems \ref{T:main_DC} and \ref{thm0}.

\paragraph{\bf Acknowledgments.} We are grateful to Michael Hutchings for a suggestion which simplified a previous proof of Proposition \ref{map2} below.
The research of A.\ Abbondandolo and B.\ Bramham is supported by the SFB/TRR 191 ``Symplectic Structures in Geometry, Algebra and Dynamics'', funded by the Deutsche Forschungsgemeinschaft. P.\ A.\ S.\ Salom\~{a}o is supported by the FAPESP grant 2011/16265-8 and the CNPq grant 306106/2016-7.

\section{A family of special area-preserving diffeomorphisms}
\label{secprima}

Throughout this article, $\omega_0$ denotes the standard area form on $\R^2$ and $\lambda_0$ its standard primitive:
\begin{equation}
\label{omegalambda}
\omega_0:=dx\wedge dy \qquad\mbox{and}\qquad \lambda_0 := \frac{1}{2} ( x\, dy - y \, dx).
\end{equation}
Note that $\lambda_0$ is invariant under rotations about the origin.  We shall often tacitly identity $\R^2$ with $\C$ by the standard identification $(x,y)\mapsto x+iy$.

\subsection{Action and Calabi invariant}

Here we recall the definition and basic facts about the action and the Calabi invariant of compactly supported area-preserving diffeomorphisms of the plane. See e.g.\ \cite[section 2.1]{abhs17b} for detailed proofs.

We denote by $\mathrm{Diff}(\R^2,\omega_0)$ the group of smooth diffeomorphisms of $\R^2$ which preserve $\omega_0$ and by $\mathrm{Diff_c}(\R^2,\omega_0)$ the subgroup consisting of compactly supported diffeomorphisms. Let $\varphi\in \mathrm{Diff_c}(\R^2,\omega_0)$ and let $\lambda$ be a smooth primitive of $\omega_0$ on $\R^2$. Since $\varphi$ preserves $\omega_0$,  the 1-form $\varphi^* \lambda - \lambda$
is closed and hence exact on $\R^2$.   The {\em action} of $\varphi$ with respect to $\lambda$ is the unique smooth function
\[
\sigma_{\varphi,\lambda} : \R^2 \rightarrow \R
\]
which is compactly supported and satifies
\[
d\sigma_{\varphi,\lambda} = \varphi^* \lambda - \lambda.
\]
Notice that $\sigma_{\varphi,\lambda}$ vanishes on any connected unbounded domain in $\R^2$ which is disjoint from the support of $\varphi$.
The next lemma describes the dependence of the action on its defining data.

\begin{lem}
\label{formule}
Let $\varphi$ and $\psi$ be elements of $\mathrm{Diff_c}(\R^2,\omega_0)$ and let $h$ be in $\mathrm{Diff}(\R^2,\omega_0)$. Let $\lambda$ be a smooth primitive of $\omega_0$ and let $u$ be a smooth real function on $\R^2$. Then:
\begin{enumerate}[(i)]
\item $\sigma_{\varphi,\lambda+du} = \sigma_{\varphi,\lambda} + u\circ \varphi - u$.
\item $\sigma_{\psi\circ \varphi,\lambda} = \sigma_{\psi,\lambda} \circ \varphi + \sigma_{\varphi, \lambda} = \sigma_{\psi,\lambda} + \sigma_{\varphi,\psi^* \lambda}$.
\item $\sigma_{\varphi^{-1},\lambda} = - \sigma_{\varphi,\lambda} \circ \varphi^{-1} = - \sigma_{\varphi,(\varphi^{-1})^*\lambda}$;
\item $\sigma_{h^{-1}\circ \varphi \circ h,h^* \lambda} = \sigma_{\varphi,\lambda}\circ h$. 
\end{enumerate}
\end{lem}

The function $\sigma_{\varphi,\lambda}$ depends on the choice of the primitive $\lambda$, but its value at fixed points of $\varphi$ does not, by Lemma \ref{formule} (i). By the same statement, also the integral of $\sigma_{\varphi,\lambda}$ is independent on the choice of the primitive $\lambda$ and defines the quantity
\[
\mathrm{CAL}(\varphi) := \int_{\R^2} \sigma_{\varphi,\lambda} \omega_0,
\]
which is called the {\em Calabi invariant} of $\varphi$. The function
\[
\mathrm{CAL}: \mathrm{Diff_c}(\R^2,\omega_0) \rightarrow \R
\]
is a surjective homomorphism of groups.

The group $\mathrm{Diff_c}(\R^2,\omega_0)$ is connected, so for every $\varphi\in \mathrm{Diff_c}(\R^2,\omega_0)$ one can find a smooth isotopy $\{\varphi^t\}_{t\in [0,1]} \subset \mathrm{Diff_c}(\R^2,\omega_0)$ such that $\varphi^0=\mathrm{id}$ and $\varphi^1=\varphi$. The fact that $\R^2$ is simply connected implies that the time dependent vector field generated by this isotopy, which is defined by
\[
X_t\bigl(\varphi^t(z)\bigr) = \frac{d}{dt} \varphi^t(z),
\]
is Hamiltonian, meaning that there exists a compactly supported smooth function 
\[
H: [0,1]\times \R^2 \rightarrow \R, \qquad H(t,z) = H_t(z),
\]
such that
\begin{equation}
\label{hvf}
\imath_{X_t} \omega_0 = dH_t.
\end{equation}
The function $H$ is called a {\em generating Hamiltonian} for $\varphi$. Conversely, a smooth compactly supported function 
\[
H: [0,1]\times \R^2 \rightarrow \R, \qquad H(t,z) = H_t(z),
\]
defines a smooth compactly supported time dependent vector field $X_{H_t}$ by the identity
\eqref{hvf}, whose flow is a smooth path in $\mathrm{Diff_c}(\R^2,\omega_0)$ starting at the identity. 

The action and the Calabi invariant of $\varphi$ can be expressed in terms of a generating Hamiltonian $H$ by the formulas
\[
\sigma_{\varphi,\lambda}(z) = \int_{t\mapsto \varphi^t(z)} \lambda + \int_0^1 H_t\bigl( \varphi^t(z) \bigr)\, dt,
\]
where $\varphi^t$ is the flow of $X_{H_t}$, and
\begin{equation}
\label{calham}
\mathrm{CAL}(\varphi) = 2\int_{[0,1]\times \R^2}  H \, dt\wedge \omega_0.
\end{equation}

\subsection{Conley-Zehnder index}\label{ssection_CZ_index}

In this subsection, we recall the definition of the Conley-Zehnder index for paths of symplectic linear automorphisms of $(\R^2,\omega_0)$ starting at the identity. We need to consider also the case in which the symplectic path is degenerate, meaning that 1 is an eigenvalue of the automorphism at the right-end point of the path. In this case, we consider the maximal lower semi-continuous extension of the Conley-Zehnder index for non-degenerate paths. This definition agrees with the one in \cite[Section 3]{hwz98}, which is the main reference for this subsection, and differs from the one in \cite{rs95}.

For every closed interval $I\subset \R$ of length strictly less than $1$ satisfying $\partial I \cap \Z = \emptyset$ define
\[
\mu(I) := \left\{ \begin{aligned} & 2k && \text{if }  I \cap \Z = \{k\} \\ & 2k+1 && \text{if} \ I \subset (k,k+1) \mbox{ for some }k\in \Z \end{aligned} \right.
\]
One extends the function $\mu$ to the set of all closed intervals $I$ of length strictly less than $1$ by the formula
\[
\mu(I) := \lim_{\delta\downarrow 0} \mu(I-\delta).
\]
Notice that in the special case in which the interval $I$ consists of only one point $a$ we have
\begin{equation}
\label{singleton}
\mu(\{a\}) = 2 \lceil a \rceil -1,
\end{equation}
where $\lceil a\rceil$ denotes the unique integer $k$ for which $a\in(k-1,k]$.

Now let $\Phi:[0,1] \mapsto {\rm Sp}(2)$ be a continuous symplectic path satisfying $\Phi(0)=\rm{id}$. For every non-zero vector $u$ in the plane choose a continuous lift $\theta_u:[0,1]\mapsto \R$ of the argument of $\Phi(t)u$. This means that
\[
\Phi(t)u = |\Phi(t)u|e^{i\theta_u(t)}.
\]
Define the number $\Delta_{\Phi}(u)$ by
\[
\Delta_{\Phi}(u) := \frac{\theta_u(1)-\theta_u(0)}{2\pi},
\]
which is independent of the choice of the lift $\theta_u$. The image $J_{\Phi}$ of the function $\Delta_{\Phi}$ is the so-called {\it rotation interval} of $\Phi$. It is an interval of length less than~$1/2$. One can show that $\partial J_{\Phi} \cap \Z = \emptyset$ if and only if $1$ is not an eigenvalue of $\Phi(1)$. We define the {\em Conley-Zehnder index} $\mu(\Phi)$ of $\Phi$ as
\[
\mu(\Phi) := \mu(J_{\Phi}).
\]
Here are two useful properties of the Conley-Zehnder index. The first one is its naturality: For any $A\in \mathrm{Sp}(2)$ we have
\begin{equation}
\label{cz1}
\mu(A^{-1} \Phi A) = \mu (\Phi).
\end{equation}
For the second one we recall that the space of free homotopy classes $[\R/\Z,\mathrm{Sp}(2)]$ is isomorphic to $\Z$ by an isomorphism called the {\em Maslov index}, which we denote by
\[
\mathrm{Maslov}: [\R/\Z,\mathrm{Sp}(2)] \rightarrow \Z,
\]
and is normalized by giving the loop of positive rotations $t\mapsto e^{2\pi it}$, $t\in \R/\Z$, the Maslov index 1. Then for any loop $\Psi:[0,1] \rightarrow \mathrm{Sp}(2)$ with $\Psi(0)=\Psi(1)=\mathrm{id}$ we have
\begin{equation}
\label{cz2}
\mu(\Psi \Phi) = \mu(\Phi) + 2 \, \mathrm{Maslov}(\Psi).
\end{equation}

\paragraph{Conley-Zehnder index of fixed points.}
Let $\varphi \in{\rm Diff_c}(\R^2,\omega_0)$ and $z_0$ be a fixed point of $\varphi$. Choose a smooth path $\{\varphi^t\}_{t\in [0,1]}\subset {\rm Diff_c}(\R^2,\omega_0)$ such that $\varphi^0=\mathrm{id}$ and $\varphi^1=\varphi$. 
The {\em Conley-Zehnder index} of $z_0$ with respect to $\varphi$ is defined to be the Conley-Zehnder index of the symplectic path
\[
[0,1] \mapsto {\rm Sp}(2), \qquad t\mapsto D\varphi^t(z_0),
\]
and is denoted by
\[
\mu(z_0,\varphi).
\]
As the notation suggests, this integer is independent of the choice of the isotopy in ${\rm Diff_c}(\R^2,\omega_0)$ which connects the identity to $\varphi$. Moreover, if $h$ is in ${\rm Diff}(\R^2,\omega_0)$ and $z_0$ is a fixed point of $\varphi$, then
\[
\mu(h(z_0),h\circ \varphi\circ h^{-1}) = \mu(z_0,\varphi).
\]
Indeed, this follows from \eqref{cz1}, \eqref{cz2} and from the fact that $\R^2$ is simply connected.

Being a fixed point of $\varphi$, $z_0$ is also a fixed point of all the iterates $\varphi^k$, and the {\em mean Conley-Zehnder index} of $z_0$ is defined as
\[
\overline{\mu}(z_0,\varphi) := \lim_{k\rightarrow +\infty} \frac{\mu(z_0,\varphi^k)}{k}.
\]
The existence of the above limit follows from the fact that the map
\[
\N \rightarrow \Z, \qquad k\mapsto \mu(z_0,\varphi^k)
\]
is a quasi-morphism.

\paragraph{Conley-Zehnder index and rotation number of closed Reeb orbits.}
Let $\beta$ be a contact form on a 3-manifold $M$, with Reeb  vector field  $R_\beta$ and Reeb flow $\phi_{R_\beta}^t$.

Given a closed Reeb orbit $(\gamma,T)$, let $\Xi$ be a $d\beta$-symplectic trivialization of the contact structure $\ker \beta$ over some neighborhood of the image of $\gamma$. We may see $\Xi$ as a smooth family of symplectic linear isomorphisms
\[
\Xi_q :(\R^2,\omega_0) \to  (\ker \beta(q),d\beta)
\]
parametrized by points $q$ near $\gamma(\R)$. A path $\Phi: [0,1] \to {\rm Sp}(2n)$ can be constructed by the formula
\[
\Phi(t) = \Xi_{\phi^{Tt}_{R_\beta}(p)}^{-1} \circ d\phi^{Tt}_{R_\beta}(p) \circ \Xi_p,
\]
where $p=\gamma(t_0)$ with $t_0\in \R$ chosen arbitrarily. The {\em Conley-Zehnder index} of $(\gamma,T)$ with respect to $\Xi$ is defined as the Conley-Zehnder index $\mu(\Phi)$ of the path $\Phi$ as defined in the beginning of section~\ref{ssection_CZ_index}, and is denoted by 
\[
\mu((\gamma,T),\Xi).
\]
As the notation suggest, this does not depend on the choice of the base point $p\in\gamma(\R)$.

If $\tilde{\Xi}$ is a different $d\beta$-symplectic trivialization of the contact structure $\ker \beta$ over some neighborhood of $\gamma(\R)$ and $\tilde{\Phi}:[0,1] \rightarrow \mathrm{Sp}(2)$ is the corresponding trivialization of the path $t\mapsto d\phi^{Tt}_{R_\beta}(p)$, we have
\[
\tilde{\Phi}(t) = \Psi(t) \Phi(t) \Psi(0)^{-1},
\]
where $\Psi: \R/\Z \rightarrow \mathrm{Sp}(2)$ is the loop
\[
\Psi(t) := \tilde{\Xi}_{\phi^{Tt}_{R_\beta}(p)}^{-1} \circ \Xi_{\phi^{Tt}_{R_\beta}(p)}.
\]
By (\ref{cz1}) and (\ref{cz2}) we get
\begin{equation}
\label{chtriv}
\mu\bigl((\gamma,T),\tilde{\Xi}\bigr) = \mu\bigl((\gamma,T),\Xi\bigr) + 2 \, \mathrm{Maslov}(\Psi).
\end{equation}
When the trivializations $\tilde{\Xi}$ and $\Xi$ are isotopic, the loop $\Psi$ is freely homotopic to a constant loop, and hence the Conley-Zehnder indices with respect to $\tilde{\Xi}$ and $\Xi$ coincide.

The {\em mean rotation number} of $(\gamma,T)$ with respect to $\Xi$ is defined as
\[
\overline{\rho}((\gamma,T),\Xi) := \frac{1}{2} \lim_{k \to \infty} \frac{\mu((\gamma,kT),\Xi)}{kT}.
\]
As in the case of the mean Conley-Zehnder index of a fixed point, the existence of the above limit follows from the fact that the map
\[
\N \rightarrow \Z, \qquad k\mapsto \mu((\gamma,kT),\Xi)
\]
is a quasi-morphism. The fact that
\[
\overline{\rho}((\gamma,T),\Xi) = \overline{\rho}((\gamma,kT),\Xi)
\]
allows us to denote the mean rotation number simply by
\[
\overline{\rho}(\gamma,\Xi) = \overline{\rho}((\gamma,T),\Xi).
\]
The trivialization $\Xi$ will be omitted from the notation for $\mu$ or $\overline{\rho}$ when it is clear from the context.

\subsection{Rotationally invariant Hamiltonians}

Rotationally invariant Hamiltonians will be useful building blocks in our construction.   In this section we compute the action, the Calabi invariant, and the Conley-Zehnder indices associated to a general autonomous radial Hamiltonian with compact support, with respect to the rotationally invariant primitive $\lambda_0$ of $\omega_0$

\begin{lem}\label{radialham}
Suppose that $H:\R^2 \rightarrow\R$ has the form $H(z)=h\big(|z|^2\big)$ for some smooth function $h:[0,+\infty) \rightarrow\R$ with compact support.
Let $\varphi^t$ be the flow of the autonomous Hamiltonian vector field $X_H$ and $\varphi = \varphi^1$ be the corresponding time-$1$ map.
Then 
\begin{equation}\label{E:radham_isotopy}
\varphi^t(z) = e^{-2 h'(|z|^2)it} z, \qquad \forall z\in \R^2, \; \forall t\in \R,
\end{equation}
and
\begin{enumerate}[(i)]
 \item $\sigma_{\varphi,\lambda_0}(z) = h(|z|^2) - |z|^2 h'(|z|^2))$ for all $z\in \R^2$.
\item $\displaystyle{\mathrm{CAL}(\varphi)= 4\pi \int_0^{+\infty} r h(r^2)\, dr}$.
 \item Let $z_0\in\Fix(\varphi^k)$, for some $k\in\N$. If $z_0\neq0$ then $kh'(|z_0|^2)/\pi \in\Z$. Moreover, the Conley-Zehnder index satisfies
\[
 \mu(z_0,\varphi^k)  =  \left\{ \begin{array}{ll} \displaystyle{- \frac{2kh'(|z_0|^2)}{\pi}}, & \mbox{ if } z_0\neq 0 \mbox{ and } h''(|z_0|^2)< 0,\\ & \\
 \displaystyle{-\frac{2kh'(|z_0|^2)}{\pi}-1}, & \mbox{ if } z_0\neq 0 \mbox{ and } h''(|z_0|^2)\geq 0,\\ & \\
\displaystyle{2\left\lceil - \frac{kh'(0)}{\pi}\right\rceil-1}, & \mbox{ if } z_0= 0.
\end{array} \right.
\]
\end{enumerate}
\end{lem}

\begin{proof}
See \cite[Lemma 2.9]{abhs17b} for the proof of (\ref{E:radham_isotopy}), (i) and (ii). Here we prove (iii). By differentiating~\eqref{E:radham_isotopy} we get the following expression for the linearized flow
\begin{equation}\label{explicit_linear_flow}
D\varphi^t(z_0)u = -4h''(|z_0|^2) \left<z_0,u\right> it \ e^{-2h'(|z_0|^2)it}z_0 + e^{-2h'(|z_0|^2)it}u
\end{equation}
where $\langle \cdot,\cdot \rangle$ denotes the Euclidean inner product on $\R^2$. Moreover, (\ref{E:radham_isotopy}) implies that $z_0\in \R^2 \setminus \{0\}$ is a fixed point of $\varphi^k$ if and only if
\[
h'\big(|z_0|^2\big)=-\frac{\pi m}{k}
\]
for some $m\in\Z$. In this case, plugging $u=iz_0$ into (\ref{explicit_linear_flow}) we get
\[
D\varphi^t(z_0)iz_0 = e^{-2h'(|z_0|^2)it}iz_0 = e^{2\pi mt/k} i z_0.
\]
Hence the integer $m$ belongs to the rotation interval of the path 
\begin{equation}
\label{path}
t\in[0,1] \mapsto D\varphi^{kt}(z_0) \in {\rm Sp}(2).
\end{equation}
Note also that since $D\varphi^k(z_0)iz_0=iz_0$, $z_0$ is a degenerate fixed point of $\varphi^k$ and hence $m$ is an endpoint of its rotation interval. Plugging $u=z_0$ and $t=k$ in \eqref{explicit_linear_flow} we get
\[
D\varphi^k(z_0)z_0 = -4h''(|z_0|^2)|z_0|^2 k iz_0 + z_0
\]
Hence the rotation interval of the path (\ref{path}), contains points in $(m,m+1)$ if $h''(|z_0|^2)<0$. In this case, we have 
\[
\mu(z_0,\varphi^k)=2m=-\frac{2kh'(|z_0|^2)}{\pi}.
\]
If $h''(|z_0|^2)\leq 0$ then the rotation interval is contained in $(m-1,m]$ and hence 
\[
\mu(z_0,\varphi^k)=2(m-1)+1=2m-1 =-\frac{2kh'(|z_0|^2)}{\pi}-1.
\]

If $z_0=0$ then from~\eqref{explicit_linear_flow} we get $D\varphi^t(z_0) = e^{-2h'(0)it}$. Hence the rotation interval of the path (\ref{path}) degenerates to the point $\frac{-kh'(0)}{\pi}$. In this case, we get from (\ref{singleton}) that 
\[
\mu(z_0,\varphi^k) = 2\left \lceil - \frac{kh'(0)}{\pi}\right \rceil-1,
\]  
as desired. \end{proof}

\subsection{Construction of a family of disk maps}

We denote by
\[
\mathbb{D}:= \big\{ z\in \R^2 \mid |z|\leq 1\big\}
\]
the closed unit disk and by ${\rm Diff_c}(\D,\omega_0)$ the subgroup of ${\rm Diff_c}(\R^2,\omega_0)$ consisting of diffeomorphisms with support in the interior of $\D$.

The proof of Theorem \ref{T:main_DC} will be based on the construction of a map $\varphi\in  {\rm Diff_c}(\D,\omega_0)$ with special properties, as described in the next proposition:

\begin{prop}
\label{map2}
For every $\epsilon>0$ there exists a smooth primitive $\lambda$ of $\omega_0$ such that $\lambda-\lambda_0$ is supported in the interior of $\D$ and a map $\varphi \in {\rm Diff_c}(\D,\omega_0)$
with the following properties: 
\begin{enumerate}[(i)]
\item The action $\sigma_{\varphi,\lambda}$ has the bounds
\[
-\frac{\pi}{2} < \sigma_{\varphi,\lambda}< \frac{\pi}{2},  
\]
with a stronger inequality at fixed points of $\varphi$: 
\[
\sigma_{\varphi,\lambda}(z_0)\geq 0 \qquad \forall z_0\in\Fix(\varphi).
\]
\item The Calabi invariant of $\varphi$ satisfies 
\[
-\frac{\pi^2}{2} < \Cal(\varphi)< -\frac{\pi^2}{2}+\epsilon.
\]
\item For each $k\in \N$ and each fixed point $z_0$ of $\varphi^k$ the Conley-Zehnder index of $z_0$ satisfies 
\[
\mu(z_0,\varphi^k)\geq  1-3k,
\]
with a stronger inequality if $z_0$ is a fixed point of $\varphi$:
\[
\mu(z_0,\varphi^k)\geq -1 \qquad \forall z_0\in \Fix(\varphi).
\]
\end{enumerate}
\end{prop}

The idea of the proof is the following. The map $\varphi$ consists of a composition $\varphi=\varphi_-\circ \varphi_+$ of two area-preserving diffeomorphisms with support in the interior of $\D$. The map $\varphi_+$ is a rotation of angle $\pi$ on a big disk contained in $\mathrm{int}(\D)$, and in the annulus between the two disks it rotates each circle centered at the origin by an angle in the interval $[0,\pi]$, which becomes 0 when the circle is contained in a small neighborhood of $\partial \D$. 

The map $\varphi_-$ is supported in a set $A$ which consists of two compact subsets $A_0$ and $A_1=-A_0$, where $A_0$ is a region diffeomorphic to a closed disk which is contained in the interior of the upper half disk $\D\cap \{\im z\geq 0\}$ and fills almost all of its area. The restrictions $\varphi_-|_{A_0}$ and $\varphi_-|_{A_1}$ are conjugated to each other by the rotation by $\pi$, and they are conjugated to a compactly supported disk map which rotates most of the disk by an angle $2\pi \theta$, with $\theta>-2$ and very close to $-2$, and then rotates less and less in the negative direction when approaching the boundary.

The maps $\varphi_+$ and $\varphi_-$ are made compatible by ensuring that the disk on which $\varphi_+$ is a rotation by $\pi$ contains the set $A$, which then implies that $\varphi_+$ and $\varphi_-$ commute.

The primitive $\lambda$ of $\omega_0$ coincides with $\lambda_0$ outside of a large disk contained in $\mathrm{int}(\D)$, and in $A_0$ and $A_1$ it coincides with the pull-back of $\lambda_0$ by the area-preserving diffeomorphisms mapping $A_0$ and $A_1$ onto a disk which conjugate $\varphi_-|_{A_0}$ and $\varphi_-|_{A_1}$ to the disk map described above.

The map $\varphi_-$ gives a negative contribution to the Calabi invariant of $\varphi$, which overrides the positive contribution given by $\varphi_+$ and makes $\mathrm{CAL}(\varphi)$ close to $-\pi^2/2$, as in statement (ii). The action of $\varphi_+$ and $\varphi_-$ with respect to $\lambda$ can be computed explicitly, and using the behaviour of the action under composition we can guarantee that the action of $\varphi$ takes values in the interval $(-\pi/2,\pi/2)$, as required by statement (i). The fact that $\varphi_-$ preserves each of the sets $A_0$ and $A_1$ while $\varphi_+$ permutes them guarantees that all fixed points of $\varphi$ lie outside of $A$. There, the map $\varphi$ coincides with $\varphi_+$, and it is easy to show that all its fixed points have non-negative action and Conley-Zehnder index not smaller than $-1$. This gives the stronger inequality in statement (i) and the second part of statement (iii). Finally, the Conley-Zehnder indices of all other periodic points can be estimated quite precisely, and in particular they satisfy the first part of statement (iii).

\medskip

In this section, we will actually prove a generalization of the above proposition, in which instead of subdividing the disk into two half-disks, we subdivide it into $n\geq 2$ equal sectors. The map $\varphi_+$ will then be a rotation of angle $2\pi/n$ on most of $\D$, while $\varphi_-$ will preserve each sector and behave as a negative rotation on a large portion of each of them. This generalization will be used for proving Theorem \ref{thm0}. We will need more information on the map $\varphi$, and in particular precise bounds relating the mean Conley-Zehnder indices of its periodic points to their action. The precise statement is the following:

\begin{prop}
\label{mapn}
Given a natural number $n \geq 2$ and a real number $\epsilon>0$ there exist a primitive $\lambda$ of $\omega_0$ such that $\lambda-\lambda_0$ is supported in the interior of $\D$ and a map $\varphi \in {\rm Diff_c}(\D,\omega_0)$ with the following properties:
\begin{enumerate}[(i)]
\item The action $\sigma_{\varphi,\lambda}$ has the bounds
\[
-\pi+\frac{\pi}{n}<\sigma_{\varphi,\lambda}<\frac{\pi}{n},
\]
with a stronger inequality at fixed points of $\varphi$:
\[
\sigma_{\varphi,\lambda}(z_0)\geq 0 \qquad \forall z_0\in \Fix(\varphi).
\]
\item The Calabi invariant of $\varphi$ satisfies
\[
-\pi^2\left(1-\frac{1}{n}\right)<\Cal(\varphi)< -\pi^2\left(1-\frac{1}{n}\right)+\epsilon.
\]
\item For each $k\in \N$ and each fixed point $z_0$ of $\varphi^k$ the Conley-Zehnder index of $z_0$ satisfies 
\[
\mu(z_0,\varphi^k)\geq -2nk+ \frac{2k}{n} + 1,
\]
with a stronger inequality if $z_0$ is a fixed point of $\varphi$:
\[
\mu(z_0,\varphi^k)\geq -1 \qquad \forall z_0\in \Fix(\varphi).
\]
\item The action $\sigma_{\varphi,\lambda}$ is invariant under $\varphi$.
\item If $z_0\in \Fix(\varphi^k)$ is not a fixed point of $\varphi$, then $k\geq n$.
\end{enumerate}
Furthermore, there is a compact set $A\subset \mathrm{int}(\D)$ which
is invariant under $\varphi$ and has the following properties:
\begin{enumerate}[(i)]
\setcounter{enumi}{5}
\item Every $z_0\in \Fix(\varphi^k)\setminus A$, $k\in \N$, satisfies
\[
\sigma_{\varphi,\lambda}(z_0) \geq 0
\]
and
\[
\frac{2}{\pi} \sigma_{\varphi,\lambda}(z_0) \leq \frac{\overline{\mu}(z_0,\varphi^k)}{k} \leq \left( \frac{2}{\pi} + \epsilon\right) \sigma_{\varphi,\lambda}(z_0).
\]

\item There exists a number $\nu\in (0,\epsilon)$ such that
every $z_0\in \Fix(\varphi^k)\cap A$, $k\in \N$, satisfies
\[
\sigma_{\varphi,\lambda}(z_0) \geq -\pi + \frac{\pi}{n} (1+\nu)
\]
and
\[
\frac{2}{n} + \frac{2n}{\pi}  \sigma_{\varphi,\lambda}(z_0)  - 2 - \nu^2 \leq \frac{\overline{\mu}(z_0,\varphi^k)}{k} \leq \frac{2}{n}  + \frac{2n}{\pi} \sigma_{\varphi,\lambda}(z_0) -2 + \nu^2.
\]
\item There exists  $w_0\in {\rm Fix}(\varphi^n) \cap A$ such that
\[
\sigma_{\varphi,\lambda}(w_0) \leq - \pi + \frac{\pi}{n} (1+\nu) + \nu
\]
and
\[
\frac{\overline{\mu}(w_0,\varphi^n)}{n} = \frac{2}{n} - 2n + 2 \nu,
\] 
where $\nu$ is the number which appears in (vii).
\end{enumerate}
\end{prop}

Notice that Proposition \ref{map2} directly follows from statements (i), (ii) and (iii) in Proposition \ref{mapn}, in which we choose $n=2$.

The remaining part of this section is devoted to the proof of Proposition \ref{mapn}. We fix once for all the natural number $n\geq 2$, which labels the 1-form $\lambda$ and the map $\varphi$ we are going to construct, and the positive real number $\epsilon$, as in the statement Proposition \ref{mapn}.

\medskip

\noindent \textsc{Construction of $\lambda$.}  For each $j\in \{0,1,\ldots,n-1\}$  we denote the closed sectors
\[
\D_{n,j} : = \left\{z\in \D  \; \Big| \;   j\frac{2\pi}{n}\leq \arg z \leq (j+1)\frac{2\pi}{n}  \right\}.
\]
For $r>0$ we write $\D_r$ for the closed disk of radius $r$,
\[
\D_r:= \{z\in \C\mid |z|\leq r \}.
\]

Here we wish to construct a suitable primitive $\lambda$ of $\omega_0$, which coincides with $\lambda_0$ outside of a compact set which is contained in the union of the interior parts of the sector $\D_{n,j}$, for $j\in \{0,\dots,n-1\}$ and is invariant under the counterclockwise rotation by the angle $2\pi/n$, which we denote by $\rho$:
\[
\rho: \R^2 \rightarrow \R^2, \qquad z\mapsto  e^{2\pi i/n} z.
\]
Choose a positive number $\eta<1$ very close to 1, and fix a smooth subset $A_0$ of the interior of the sector $\D_{n,0}$ diffeomorphic to a closed disk and having area $\eta \pi/n$. Fix moreover an area preserving diffeomorphism $f\in {\rm Diff}(\R^2,\omega_0)$ which maps $A_0$ to a disk centered at the origin of some radius $R>0$,
\[
f(A_0)=\D_{R}.
\]
Since $f$ is area preserving,
\begin{equation}\label{E:R}
	 \pi R^2 = \eta \frac{\pi}{n}.
\end{equation}
Since $f^*\lambda_0$ is a primitive of $\omega_0$, there is a smooth
function $u:\R^2 \rightarrow\R$ satisfying
\[
f^*\lambda_0=\lambda_0 + du.
\]
Now fix a smooth cut-off function $\chi:\R^2 \rightarrow[0,1]$ such that $\chi|_{A_0}\equiv1$ and whose support  is contained in the interior of $\D_{n,0}$.     Then set
\begin{equation}\label{E:lambda}
\lambda := \lambda_0 +d\hat u,
\end{equation}
where $\hat u:\R^2 \to \R$ is given by
\begin{equation}\label{uchapeu}
\hat u = \chi u + \rho^*(\chi u) + \ldots +(\rho^{n-1})^*(\chi u).
\end{equation}

The smooth $1$-form $\lambda$, which is a primitive of $\omega_0$, has the following properties:
\begin{enumerate}[(a)]
\item $\rho^*\lambda=\lambda$;
\item $\lambda=f^*\lambda_0$ on $A_0$;
\item $\lambda=\lambda_0$ outside of a compact set which is contained in the union of the interiors of the sectors $\D_{n,j}$, $j\in \{0,\dots,n-1\}$.
\end{enumerate}

The positive number $\eta<1$ which appears in the construction will be made sufficiently close to 1 in due time, depending on $n$ and $\epsilon$.

\medskip

\noindent{\sc Construction of the map $\varphi^+$.} We will construct the map $\varphi \in{\rm Diff_c}(\D,\omega_0)$ as a composition of two maps $\varphi_+,\varphi_- \in {\rm Diff_c}(\D,\omega_0)$
which are each generated by autonomous Hamiltonians $H_+$ and $H_-$ supported in the interior of the disk that separately we understand well. Here  we define $\varphi_+$ and compile its basic properties.

We start by fixing a family of smooth cut-off functions $\chi_\delta: [0,+\infty) \rightarrow \R$ depending on $\delta \in (0,1/2)$: Let $\chi_\delta$ be a smooth convex function which is supported in $[0,1)$ and satisfies
\begin{equation}
\label{chi}
\chi_\delta(s) = 1 - \delta - s \qquad \forall s\in [0,1-2\delta].
\end{equation}
It follows that $\chi_\delta$ is monotonically decreasing, non-negative, and satisfies
\begin{equation}\label{chi2}
\begin{aligned}
 \max\{1-\delta-s,0\} \leq &\chi_\delta(s) \leq  \max\{ (1 - \delta)(1 - s),0\}, \\
-1 &\leq \chi_\delta'(s) \leq 0,
\end{aligned}
\end{equation}
for every $s\in [0,+\infty)$, and
\begin{equation}\label{chi3}
\begin{aligned}
  0\leq \chi_\delta(s) - s \chi_\delta'(s) \leq 1-\delta &  &\forall s\in [0,+\infty)\\
    \chi_\delta(s) - s \chi_\delta'(s) = 1-\delta &  &\quad\forall s\in [0,1-2\delta]
\end{aligned}
\end{equation}
where the last inequalities follow from the fact that the function $\chi(s) - s \chi_\delta'(s)$ takes the value $1-\delta$ for $s=0$ and $0$ for
$s\geq 1$, and has derivative $-s\chi_\delta''(s)\leq 0$ which vanishes on $[0,1-2\delta]$. Using this same type of argument one can show moreover that
\begin{equation}\label{chi_estimate}
-(1-\delta)\chi_\delta'(s) \leq \chi_\delta(s) - s\chi_\delta'(s) \leq -\chi_\delta'(s) \ \ \forall s\in[0,+\infty).
\end{equation}

Choose $\delta>0$ to be so small that
\begin{equation}
\label{dovedelta}
\mathrm{supp} (\hat{u}) \subset \mathrm{int}\bigl(\D_{\sqrt{1-2\delta}}\bigr),
\end{equation}
where $\hat{u}$ is the smooth function which is defined in \eqref{uchapeu}. Notice that the above assumption requires $\delta$ to go to zero as $\eta$ approaches 1.
Later on we will require $\delta$ to be even smaller, but we will always keep the above requirement. The autonomous Hamiltonian $H_+: \R^2 \rightarrow \R$ defined by
\[
H_+(z) = h_+(|z|^2) \qquad \mbox{with} \qquad h_+(s):= \frac{\pi}{n} \chi_{\delta}(s),
\]
is supported in $\mathrm{int}(\D)$. We denote by $\varphi_+^t\in {\rm Diff_c}(\D,\omega_0)$ the flow of $X_{H_+}$ and by $\varphi_+:=\varphi_+^1$ the time-1 map.
By Lemma~\ref{radialham} the diffeomorphism $\varphi_+^t$ restricts to the counterclockwise rotation of angle $2\pi t/n$ on the disk $\D_{\sqrt{1-2\delta}}$.   Outside of this disk, the map $\varphi_+^t$ rotates each circle about the origin counterclockwise by a non-negative angle which does not exceed $2\pi t/n$ (because of (\ref{chi2})), and which becomes zero outside of some disk of radius smaller than 1. Again by Lemma~\ref{radialham}, the action of $\varphi_+$
with respect to $\lambda_0$ is the radial function
\begin{equation}\label{sigmaphi+}
\sigma_{\varphi_+,\lambda_0}(z) = \frac{\pi }{n} \Bigl( \chi_{\delta}(|z|^2) -  |z|^2\chi_{\delta}'(|z|^2) \Bigr)\qquad\forall z\in\R^2.
\end{equation}
The fact that $\sigma_{\varphi_+,\lambda_0}(z)$ depends only on $|z|$ and the fact that $\varphi_+$ preserves  each circle centred at the origin imply that
\begin{equation}
\label{cons+0}
\sigma_{\varphi_+,\lambda_0} \circ \varphi_+ = \sigma_{\varphi_+,\lambda_0}.
\end{equation}
By (\ref{chi3}), the function $\sigma_{\varphi_+,\lambda_0}$ satisfies
\begin{align}
\label{azphi+1}
\sigma_{\varphi_+,\lambda_0}(z) &= \frac{\pi }{n} (1-\delta) \qquad \forall z\in\D_{\sqrt{1-2\delta}}, \\
\label{azphi+2}
0 \leq \sigma_{\varphi_+,\lambda_0}(z) &\leq \frac{\pi }{n} (1-\delta)   \qquad   \forall  z\in \R^2.
\end{align}

Integrating the action on $\D$, we find that the Calabi invariant of $\varphi_+$ has the upper bound
\begin{equation}\label{calH+}
{\rm CAL}(\varphi_+) \leq \frac{\pi^2}{n}(1-\delta) < \frac{\pi^2}{n}.
\end{equation}
Now we wish to determine the action of $\varphi_+$ with respect to the primitive $\lambda$ from (\ref{E:lambda}).
Recall that $\lambda= \lambda_0 + d\hat u,$ where $\hat u:\R^2 \to \R$, which  is given by \eqref{uchapeu}. Observe that $\hat u$ is invariant under $\varphi_+$. Indeed, this follows from the invariance of $\hat{u}$ under $\rho$, because $\varphi_+=\rho$ on $\D_{\sqrt{1-2\delta}}$ and $\hat u=0$ on $\R^2 \setminus \D_{\sqrt{1-2\delta}}$ thanks to \eqref{dovedelta}. Hence, by Lemma \ref{formule} (i), we have
\begin{equation}\label{llambda}
\sigma_{\varphi_+,\lambda} = \sigma_{\varphi_+,\lambda_0} \qquad \mbox{ on } \R^2.
\end{equation}
Therefore, we can rewrite \eqref{cons+0} as
\begin{equation}
\label{cons+}
\sigma_{\varphi_+,\lambda} \circ \varphi_+ = \sigma_{\varphi_+,\lambda}, 
\end{equation}
and (\ref{azphi+1}) and \eqref{azphi+2} as
\begin{align}
\label{E:action_Hplus_interior}
\sigma_{\varphi_+,\lambda}(z)  &=  \frac{\pi}{n}(1-\delta) \qquad \forall z\in \D_{\sqrt{1-2\delta}},
\\
\label{E:action_Hplus_boundary}
0 \leq \sigma_{\varphi_+,\lambda}(z) &\leq \frac{\pi}{n} (1-\delta)   \qquad\forall z\in \R^2.
\end{align}

Now we wish to establish suitable bounds on the Conley-Zehnder indices of the periodic points of $\varphi_+$. By using Lemma \ref{radialham} (iii) and the fact that $h_+''\geq 0$, we find that the Conley-Zehnder index of a point $z_0\in\Fix(\varphi_+^k)\setminus \{0\}$, $k\in \N$, has the value
\begin{equation}
\label{cczz1}
\mu(z_0,\varphi_+^k)= -\frac{2kh'_+(|z_0|^2)}{\pi}-1= -\frac{2k}{n}\chi'_\delta(|z_0|^2)-1,
\end{equation}
and hence
\begin{equation}\label{E:CZ_Hplus0}
-1\leq \mu(z_0,\varphi_+^k) \leq \frac{2k}{n}-1, \qquad \forall z_0\in \Fix(\varphi_+^k)\setminus \{0\},
\end{equation} 
because $-1\leq \chi_{\delta}'\leq 0$.
Moreover, the fixed point $0$ of $\varphi$ has Conley-Zehnder index
\begin{equation}\label{E:CZ_Hplus}
\mu(0,\varphi_+^k)= 2\left\lceil -\frac{k h_+'(0)}{\pi} \right\rceil-1 = 2\left\lceil -\frac{k}{n}\chi'_\delta(0)\right\rceil-1 =2\left\lceil \frac{k}{n}\right\rceil-1,
\end{equation} because $\chi'_\delta(0)=-1$.

We conclude our study of the map $\varphi_+$ by estimating the mean Conley-Zehnder index of its periodic points in terms of their action. We start with a fixed point $z_0$ of $\varphi_+^k$ other than $0$. For any $h\in \N$ the identity \eqref{cczz1} applied to $z_0\in \mathrm{Fix}(\varphi_+^{hk})\setminus \{0\}$ gives us
\[
\mu(z_0,\varphi_+^{hk}) = - \frac{2hk}{n} \chi_{\delta}'(|z_0|^2) -1,
\]
and dividing by $hk$ and taking a limit for $h\rightarrow +\infty$ we obtain
\begin{equation}
\label{mcczz}
\frac{\overline{\mu}(z_0,\varphi^k_+)}{k} = - \frac{2}{n} \chi_{\delta}'(|z_0|^2).
\end{equation}
In the case of the fixed point $z_0=0$ we have by \eqref{E:CZ_Hplus}
\[
\frac{\overline{\mu}(0,\varphi_+^k)}{k} = \frac{2}{n} =   - \frac{2}{n} \chi_{\delta}'(0),
\]
and hence the formula (\ref{mcczz}) holds for all $z_0\in \Fix(\varphi^k_+)$.

Let $z_0$ be a fixed point of $\varphi^k_+$.
Using \eqref{llambda}, \eqref{sigmaphi+} and \eqref{chi_estimate}, we obtain the inequality
\[
-(1-\delta)\frac{\pi}{n}  \chi_\delta'\left(|z_0|^2 \right) \leq  \sigma_{\varphi_+,\lambda}(z_0)\leq -\frac{\pi}{n}  \chi_\delta'\left(|z_0|^2 \right) ,
\] 
which can be restated as
\[
\frac{n}{\pi} \sigma_{\varphi_+,\lambda}(z_0)\leq - \chi_{\delta}' (|z_0|^2) \leq \frac{n}{\pi(1-\delta)} \sigma_{\varphi_+,\lambda}(z_0).
\]
Then \eqref{mcczz} implies the bounds
\[
\frac{2}{\pi} \sigma_{\varphi_+,\lambda}(z_0) \leq \frac{\overline{\mu}(z_0,\varphi^k)}{k} \leq \frac{2}{\pi(1-\delta)}  \sigma_{\varphi_+,\lambda}(z_0).
\]
By choosing $\delta$ small enough we conclude that
\begin{equation}
\label{mcz+}
\frac{2}{\pi} \sigma_{\varphi_+,\lambda}(z_0) \leq \frac{\overline{\mu}(z_0,\varphi^k)}{k} \leq \left( \frac{2}{\pi}+\epsilon \right)  \sigma_{\varphi_+,\lambda}(z_0) \qquad \forall z_0\in \mathrm{Fix}(\varphi_+^k).
\end{equation}

\medskip

\noindent \textsc{Construction of the map $\varphi_K$.} The map $\varphi_-$ will be defined by starting from a map $\varphi_K$ which is induced by an autonomous radial Hamiltonian $K$, which we now wish to define. First we fix some real number $\theta$ satisfying
\begin{equation}\label{E:theta}
			-n< \theta <-n+1,
\end{equation}
and close enough to $-n$, as we will later specify.
Recall that $R>0$ was fixed above (\ref{E:R}), so that there is an area-preserving diffeomorphism $f:A_0\rightarrow\D_R$ from some
$A_0\subset \mathrm{int}(\D_{n,0})$. Given $\delta\in (0,1/2)$ as in the constriction of $\varphi_+$, consider the following rotationally invariant Hamiltonian 
\[
K:\C\rightarrow\R,\qquad K(z)=h_K(|z|^2)\quad \mbox{with} \qquad h_K(s) := \pi R^2\theta\chi_{\delta}\left(\frac{s}{R^2}\right),
\]
which is supported in the interior of $\D_R$. 
Let $\varphi_{K}^t$ denote the flow of $K$ and let $\varphi_K:=\varphi_K^1$ be the time-$1$ map. For any $t\in[0,1]$, the map $\varphi^t_K$ rotates each concentric circle about the origin by some angle in the interval 
\[
[2\pi\theta,0]\subset(-2\pi n,0],
\]
since 
\[
h_K'(|z|^2) = \pi \theta \chi_\delta' (|z|^2/R^2 ) \qquad \mbox{and} \qquad -1\leq \chi_\delta'\leq 0,
\]
see Lemma \ref{radialham}.
By Lemma \ref{radialham} (i),
\begin{equation}\label{sigmaphik}
\begin{aligned}
\sigma_{\varphi_K,\lambda_0}(z) & = h_K(|z|^2) -|z|^2h_K'(|z|^2) \\
 & = \pi R^2\theta \left( \chi_\delta\left(\frac{|z|^2}{R^2}\right) - \frac{|z|^2}{R^2}\chi'_\delta\left(\frac{|z|^2}{R^2}\right)\right) \qquad \forall z\in\R^2,
\end{aligned}
\end{equation}
and in particular
\begin{equation}
\label{consK} 
\sigma_{\varphi_K,\lambda_0} \circ \varphi_K = \sigma_{\varphi_K,\lambda_0}.
\end{equation}
By \eqref{sigmaphik} and (\ref{chi3}) we have
\begin{align*}
\pi R^2\theta(1-\delta) \leq & \ \sigma_{\varphi_K,\lambda_0}(z)\leq 0 & &\forall z \in \R^2, \\
\sigma_{\varphi_K,\lambda_0}(z) &= \pi R^2\theta(1-\delta) & &\forall z\in \D_{R\sqrt{1-2\delta}}.
\end{align*}
Together with \eqref{E:R} we get
\begin{equation}\label{E:action_K}
\begin{aligned}
\eta \frac{\pi}{n} \theta (1-\delta) \leq & \sigma_{\varphi_K,\lambda_0}(z)\leq 0 &  &\forall z\in\R^2, \\
 & \sigma_{\varphi_K,\lambda_0}(z)= \eta \frac{\pi}{n} \theta (1-\delta) &  &\forall z\in \D_{R\sqrt{1-2\delta}}.
\end{aligned}
\end{equation}
Integrating the non-positive function $\sigma_{\varphi_K,\lambda_0}$ over $\D_R$ we find the following upper bound on the Calabi invariant of $\varphi_K$:
\[
\Cal(\varphi_K) \leq \eta \frac{\pi}{n} \theta (1-\delta) \pi R^2 (1-2\delta) = \eta^2 \frac{\pi^2}{n^2} \theta (1-\delta)(1-2\delta).
\]
When $\theta$ approaches $-n$, $\eta$ approaches 1 and $\delta$ goes to $0$, the quantity on the right-hand side converges to $-\pi^2/n$. Therefore, by choosing $\theta>-n$ close enough to $-n$, $\eta<1$ close enough to 1 and $\delta>0$ close enough to $0$ we obtain the bound
\begin{equation}
\label{E:Cal_K}
\Cal(\varphi_K) < - \frac{\pi^2}{n} + \frac{\epsilon}{n}.
\end{equation}
The number $\theta$ is now fixed by the above requirement, together with
\begin{equation}
\label{extra}
\theta  < -n + \epsilon.
\end{equation}
The distances of $\eta$ from $1$ and the number $\delta$ will be made even smaller below.

We now estimate the Conley-Zehnder indices of the periodic points of $\varphi_K$.  Let $k\in \N$. By Lemma \ref{radialham} (iii) we have
\begin{equation}
\label{cz0}
\mu(0,\varphi_K^k) =  2\left\lceil - \frac{kh'_K(0)}{\pi}\right\rceil-1 = 2\bigl\lceil - k \theta \chi'_{\epsilon}(0) \bigr\rceil -1 = 2\lceil k \theta\rceil -1.
\end{equation}
Now $\theta>-n$ implies $k\theta>-kn$, from which $\lceil k\theta \rceil\geq -kn+1$ and hence
\[
2\lceil k \theta\rceil -1 \geq -2kn + 1.
\]
Together with the inequality $k\theta<k(-n+1)\leq -k \leq -1$ we deduce that
\[
-2kn + 1\leq  \mu(0,\varphi_K^k) \leq -1.
\]
Now let $z_0\in \Fix(\varphi_K^k)\setminus \{0\}$. By Lemma \ref{radialham} (iii), the number $k\theta \chi_{\delta}'(|z_0|^2/R^2)$ is an integer, and either
\begin{equation}
\label{cczz2}
\mu(z_0,\varphi_K^k) = - \frac{2kh'_K(|z_0|^2)}{\pi} = - 2k\theta\chi'_{\delta}\left(\frac{|z_0|^2}{R^2}\right),
\end{equation}
if $\chi_{\delta}''(|z_0|^2/R^2)>0$, or
\begin{equation}
\label{cczz3}
\mu(z_0,\varphi_K^k) =- 2k\theta\chi'_{\delta}\left(\frac{|z_0|^2}{R^2}\right)-1,
\end{equation}
if $\chi_{\delta}''(|z_0|^2/R^2)=0$. In both cases, the bounds $0\geq \chi_{\delta}'\geq -1$ and $0>\theta>-n$ imply $\theta \chi_{\delta}' <n$ and hence 
\[
\mu(z_0,\varphi_K^k) \geq - 2 \left\lceil k\theta\chi'_{\delta}\Bigl(\frac{|z_0|^2}{R^2}\Bigr)\right\rceil- 1 \geq  - 2kn +1.
\]
The integer $k\theta\chi'_{\delta}(|z_0|^2/R^2)$ appearing in (\ref{cczz2}) and (\ref{cczz3}) is non-negative. When this integer is positive, we get
\[
\mu(z_0,\varphi_K^k) \leq -2.
\]
When it is zero, we have $\chi_{\delta}'(|z_0|^2/R^2)=0$ and hence $\chi_{\delta}''(|z_0|^2/R^2)=0$. This implies that the alternative (\ref{cczz3}) holds and therefore
\[
\mu(z_0,\varphi_K^k) = -1.
\]
We conclude that in all cases
\[
\mu(z_0,\varphi_K^k) \leq -1.
\]
We can summarize the above discussion by stating the following bounds on the Conley-Zehnder index of an arbitrary fixed point of $\varphi_K^k$:
\begin{equation}\label{E:CZ_K}
-2nk +1\leq \mu(z_0,\varphi_K^k)\leq -1 \qquad\forall z_0\in\Fix(\varphi_K^k).
\end{equation}

We conclude the analysis of $\varphi_K$ by obtaining bounds for the mean Conley-Zehnder index of $z_0\in \Fix(\varphi_K^k)$ in terms of its action. By \eqref{cz0}, \eqref{cczz2} and \eqref{cczz3} the mean Conley-Zehnder index of a periodic point is given by the formula
\begin{equation}
\label{ccczzz}
\frac{\overline{\mu}(z_0,\varphi^k_K)}{k} = - 2 \theta \chi_{\delta}' \left( \frac{|z_0|^2}{R^2} \right) \qquad \forall z_0\in \Fix (\varphi_K^k).
\end{equation}
Using \eqref{sigmaphik} and \eqref{chi_estimate} we obtain
\[
-(1-\delta)\pi R^2\theta \chi_\delta'\left(\frac{|z_0|^2}{R^2} \right) \geq  \sigma_{\varphi_K,\lambda_0}(z')\geq -\pi R^2\theta \chi_\delta'\left(\frac{|z_0|^2}{R^2} \right),
\]
which we can restate as
\[
- \frac{\sigma_{\varphi_K,\lambda_0}(z')}{\pi R^2 \theta (1-\delta)}
 \leq \chi_\delta'\left(\frac{|z_0|^2}{R^2} \right) \leq - \frac{\sigma_{\varphi_K,\lambda_0}(z')}{\pi R^2 \theta}.
\]
Then \eqref{ccczzz} implies
\[
\frac{2}{\pi R^2 (1-\delta)} \sigma_{\varphi_K,\lambda_0}(z_0)
 \leq\frac{\overline{\mu}(z_0,\varphi^k_K)}{k} \leq \frac{2}{\pi R^2}  \sigma_{\varphi_K,\lambda_0}(z_0).
\]
By \eqref{E:R} we have $\pi R^2=\eta\pi/n<\pi/n$, and we conclude that
\begin{equation}
\label{mczK}
\frac{2n}{\eta\pi(1-\delta)}  \sigma_{\varphi_K,\lambda_0}(z_0) \leq \frac{\overline{\mu}(z_0,\varphi^k_K)}{k} \leq \frac{2n}{\pi} \sigma_{\varphi_K,\lambda_0}(z_0) \qquad \forall z_0\in \Fix(\varphi^k).
\end{equation}

\medskip

\noindent \textsc{Construction of the map $\varphi_-$.} Now we are ready to define $\varphi_-\in {\rm Diff_c}(\D,\omega_0)$.   Recall that $f:A_0\rightarrow\D_R$ is an area preserving diffeomorphism from $A_0\subset\textup{int}(\D_{n,0})$. The set $\rho^j(A_0)$ is contained in the interior of the sector $\D_{n,j}$ for every $ j\in \{0,\ldots,n-1\}$, and, in particular, these sets are mutually disjoint. Set
\[
A:= \bigcup_{j=0}^{n-1}\rho^j(A_0).
\]
Define the autonomous Hamiltonian $H_-:\R^2 \rightarrow\R$ by
\[
H_-(z):=\left\{ \begin{aligned}
&K\circ f(\rho^{-j}(z)) &   &\forall z\in \rho^j(A_0),j\in \{0,1,\ldots,n-1\},\\
& 0            &    &\forall z\in \R^2 \setminus A,
\end{aligned}\right.
\]
which is smooth and supported in $A\subset \mathrm{int}(\D)$, because $K:\R^2 \rightarrow\R$ is supported in the interior of $\D_R$. We denote the flow of $H_-$ by $\varphi_-^t\in {\rm Diff_c}(\D,\omega_0)$, and its time-1 map by $\varphi_-=\varphi_-^1$.  By construction, $\varphi_-$ is supported in $A$ and
\[
\varphi_-|_{A_0} = f^{-1} \circ \varphi_K \circ f|_{A_0}.
\]
Moreover, $\varphi_-$ commutes with $\rho$:
\[
\varphi_- \circ \rho = \rho \circ \varphi_-.
\]
With the help of (\ref{calham}) the Calabi invariant of $\varphi_-$ is easily calculated to be
\begin{align*}
\Cal(\varphi_-) &= 2\int_\C H_-\, dx\wedge dy = 2\sum_{k=0}^{n-1}\int_{\rho^k(A)} H_-\, dx\wedge dy \nonumber \\
&= 2n\int_A H_-\, dx\wedge dy =2n\int_{A} K\circ f\, dx\wedge dy \nonumber \\
&= 2n\int_{\D_R} K\, dx\wedge dy = 2n\int_{\C} K\, dx\wedge dy = n\cdot\Cal(\varphi_K),
\end{align*}
where we have used the facts that $\rho$ and $f$ are area-preserving and $H_-$ is invariant under $\rho$. By \eqref{E:Cal_K} we find the upper bound
\begin{equation} 
\label{Calabi_varphi_-}
\Cal(\varphi_-) < -\pi^2 + \epsilon.
\end{equation}

Let us estimate the action function of $\varphi_-$ with respect to the special primitive~$\lambda$. The flow $\varphi^t_-$ fixes all points in the connected set $\C\setminus A$, so the action function $\sigma_{\varphi_-,\lambda}$ at such points vanishes. At points in $A_0$, using Lemma \ref{formule} (iv) we find,
\begin{equation}\label{phikphimenos}
\sigma_{\varphi_{-},\lambda}(z)=\sigma_{\varphi_{K},\lambda_0}\big(f(z)\big)\qquad\forall z\in A_0.
\end{equation}
Again, using Lemma \ref{formule} (iv) and the fact that $\lambda$ is $\rho$-invariant, one finds that  $\sigma_{\varphi_{-},\lambda}$ is invariant under $\rho$:
\begin{equation}
\label{sigma-invrho}
\sigma_{\varphi_-,\lambda} \circ \rho = \sigma_{\varphi_-,\lambda},
\end{equation}
and by using \eqref{consK} that $\sigma_{\varphi_{-},\lambda}$ is invariant under $\varphi_-$:
 \begin{equation}
 \label{cons-}
 \sigma_{\varphi_-,\lambda} \circ \varphi_- = \sigma_{\varphi_-,\lambda}.
\end{equation}
By \eqref{sigma-invrho} we find, using also our estimates (\ref{E:action_K}) on the action of $\varphi_K$ with respect to $\lambda_0$, that \begin{eqnarray}
\label{E:action_A1}
-\pi(1-\delta)< \eta \frac{\pi}{n} \theta (1-\delta) \leq  \sigma_{\varphi_-,\lambda}(z) \leq 0&   \qquad&\forall z\in A, \\
\label{E:action_A2}
 \sigma_{\varphi_-,\lambda}(z) = 0&   \qquad &\forall z\in \R^2 \setminus A.
\end{eqnarray}
Now let us consider the Conley-Zehnder indices at periodic points of $\varphi_-$.   Suppose that
\[
z_0\in\Fix(\varphi_-^k)
\]
for some $k\in\N$. If $z_0\in \rho^j(A_0)$ then the invariance of the Conley-Zehnder index under conjugacy implies
\begin{equation}
\label{mu=mu}
\mu(z_0,\varphi_-^k)=\mu(f(\rho^{-j}(z_0)),\varphi_K^k).
\end{equation}
Combining this with (\ref{E:CZ_K}) we obtain
\begin{equation}\label{CZ:Hminus_A}
-2nk + 1 \leq \mu(z_0,\varphi_-^k)\leq -1\qquad \forall z_0\in A.
\end{equation}
All points in $\R^2\setminus A$ are fixed by $\varphi_-$. Therefore
\begin{equation}\label{muphi_-} 
\mu(z_0,\varphi_-^k) = -1 \qquad \forall z_0 \in \R^2 \setminus A.\end{equation}
By \eqref{phikphimenos}, the invariance of $\sigma_{\varphi_-,\lambda}$ under $\rho$ and \eqref{mu=mu}, the bounds (\ref{mczK}) on the mean Conley-Zehnder indices of the periodic points of $\varphi_K$ translate into
\begin{equation}
\label{mcz-}
\frac{2n}{\eta\pi(1-\delta)}  \sigma_{\varphi_-,\lambda_0}(z_0) \leq \frac{\overline{\mu}(z_0,\varphi^k_-)}{k} \leq \frac{2n}{\pi} \sigma_{\varphi_-,\lambda_0}(z_0)
 \qquad \forall z_0\in \Fix(\varphi^k_-)\cap A.
\end{equation}

\textsc{Construction of the map $\varphi$.} We start by observing that the maps $\varphi^+$ and $\varphi^-$ commute:
\begin{equation}\label{commute}
\varphi_+ \circ \varphi_- = \varphi_- \circ \varphi_+
\end{equation}
Indeed, since $\varphi_+=\rho$ on $\D_{\sqrt{1-2\delta}}$, $\varphi_-$ commutes with $\rho$ and 
\[
\supp(\varphi_-) \subset A\subset \D_{\sqrt{1-2\delta}},
\]
we get
\begin{align*}
& z\in \D_{\sqrt{1-2\delta}} && \Rightarrow && \varphi_+(\varphi_-(z)) = \rho(\varphi_-(z)) = \varphi_-(\rho(z)) = \varphi_-(\varphi_+(z)), \\
& z\not\in \D_{\sqrt{1-2\delta}} && \Rightarrow && \varphi_+(\varphi_-(z)) = \varphi_+(z) = \varphi_-(\varphi_+(z)),
\end{align*}
as desired.

The map $\varphi\in {\rm Diff_c}(\D,\omega_0)$ is defined to be the composition
\[
\varphi:= \varphi_+ \circ \varphi_- = \varphi_- \circ \varphi_+.
\]
From the homomorphism property of the Calabi invariant
\[
		\Cal(\varphi)=\Cal(\varphi_-)+\Cal(\varphi_+)
\]
and so by (\ref{calH+}) and \eqref{Calabi_varphi_-} we obtain the upper bound
\begin{equation}
\label{calalto}
\Cal(\varphi) < \frac{\pi^2}{n}- \pi^2 + \epsilon = - \pi^2 \left( 1 - \frac{1}{n} \right) + \epsilon.
\end{equation}

Since $\varphi_+=\rho$ on the disk $\D_{\sqrt{1-2\delta}}$, which contains the support of $\sigma_{\varphi_-,\lambda}$,  \eqref{sigma-invrho} implies
\begin{equation}
\label{questo}
\sigma_{\varphi_-,\lambda}\circ \varphi_+ = \sigma_{\varphi_-,\lambda}.
\end{equation}
Together with Lemma \ref{formule} (ii) this implies the following formula for the action of $\varphi$ with respect to $\lambda$:
\begin{equation}\label{action_varphi_breakup}
\sigma_{\varphi,\lambda}= \sigma_{\varphi_- \circ \varphi_+,\lambda} = \sigma_{\varphi_-,\lambda}\circ \varphi_+ + \sigma_{\varphi_+,\lambda} = \sigma_{\varphi_-,\lambda} + \sigma_{\varphi_+,\lambda}.
\end{equation}
Moreover, the action is $\varphi$-invariant:
\begin{equation}
\label{cons}
\sigma_{\varphi,\lambda}\circ \varphi = \sigma_{\varphi,\lambda}.
\end{equation}
Indeed, by \eqref{cons+}, \eqref{cons-} and \eqref{questo} we have
\[
\begin{split}
\sigma_{\varphi,\lambda}\circ \varphi &= \sigma_{\varphi_-,\lambda} \circ \varphi_- \circ \varphi_+ + \sigma_{\varphi_+,\lambda} \circ \varphi_+ \circ \varphi_- \\ &= \sigma_{\varphi_-,\lambda}  \circ \varphi_+ + \sigma_{\varphi_+,\lambda} \circ \varphi_- \\ &= \sigma_{\varphi_-,\lambda} + \sigma_{\varphi_+,\lambda} \circ \varphi_- = \sigma_{\varphi_+\circ \varphi_-,\lambda} = \sigma_{\varphi,\lambda},
\end{split}
\]
where in the last but one equality we have used Lemma \ref{formule} (ii). Therefore, we have shown that the action $\sigma_{\varphi,\lambda}$ is $\varphi$-invariant, as claimed in statement (iv).

Thanks to \eqref{E:action_Hplus_boundary} and \eqref{E:action_A2}, the identity \eqref{action_varphi_breakup} gives 
\begin{equation}
\label{azfuori}
0\leq \sigma_{\varphi,\lambda}(z)	<\frac{\pi}{n}  \qquad\forall z\in \R^2 \setminus A.
\end{equation}
By (\ref{E:action_Hplus_interior}) and (\ref{E:action_A1})  we also have
\begin{equation}
\label{stimaz}
\left( - \pi + \frac{\pi}{n} \right) (1-\delta) < \frac{\pi}{n} (\eta \theta + 1) (1-\delta) \leq 
\sigma_{\varphi,\lambda}(z) \leq \frac{\pi}{n}(1-\delta) \qquad\forall z\in A,
\end{equation}
where in the first inequality we have used $\eta \theta > -\eta n > -n$.
In particular, the above two bounds imply 
\begin{equation}
\label{actionbound}
-\pi +\frac{\pi}{n} <	\sigma_{\varphi,\lambda}(z)	<\frac{\pi}{n} \qquad \forall z\in \R^2,
\end{equation}
which proves the first part of statement (i).
Now suppose that $z_0$ is a fixed point of $\varphi$.  Then $z_0$ is not in $A$, because  $\varphi_-$ preserves each component $\rho^j(A_0)$, while $\varphi_+$ maps $\rho^j(A_0)$ to $\rho^{j+1}(A_0)$ for every $j=0,\ldots, n-1$.  From (\ref{azfuori}) we have
\[						
\sigma_{\varphi,\lambda}(z_0)\geq0 \qquad \forall z_0\in \Fix(\varphi),
\]
concluding the proof of statement (i).

By integrating the first inequality of \eqref{actionbound} we get
\[
\Cal(\varphi) = \int_{\R^2} \sigma_{\varphi,\lambda}\, \omega_0 = \int_{\D} \sigma_{\varphi,\lambda}\, \omega_0 > -\pi^2 + \frac{\pi^2}{n}.
\]
Together with \eqref{calalto}, this proves the bounds on the Calabi invariant of $\varphi$ which are stated in (ii).

The fact that $\varphi_+$ cyclically permutes the components $A_j$ of $A$ and the fact that $\varphi_-$ preserves each of them imply that the periodic points which are contained in $A$ have period multiple of $n$. Since $\varphi=\varphi_+$ on $\C\setminus A$, the fact that all the periodic points of $\varphi_+$ which are not fixed points have period at least $n$ implies that the same is true for the periodic points of $\varphi$ which are not fixed by $\varphi$. We conclude that if $z_0$ is in $\Fix(\varphi^k)$ for some $k\in \N$ but not in $\Fix (\varphi)$, then $k\geq n$. This proves statement (v).

We now estimate the Conley-Zehnder indices of the periodic points of $\varphi$. We start by the fixed points of $\varphi$, which as we have seen belong to the complement of $A$, on which $\varphi=\varphi_+$. A fixed point $z_0$ of $\varphi_+$ is either outside of the support of $\varphi_+$ or $z_0=0$. In the first case, 
\[
\mu(z_0,\varphi^k)= \mu(z_0,\varphi^k_+) = \mu(z_0,{\rm id})=-1,
\]
and in the second case
\[
\mu(0,\varphi^k)= \mu(0,\varphi^k_+) = 2\left\lceil\frac{k}{n} \right \rceil-1,
\]
by \eqref{E:CZ_Hplus}. In either case, we have
\[
\mu(z_0,\varphi^k) \geq -1 \qquad \forall z_0\in \Fix(\varphi), \; \forall k\in \N,
\]
proving the second part of statement (iii).

The study of an arbitrary periodic point is simplified by the fact that, by \eqref{commute},
\[
\varphi^k = \varphi_+^k \circ \varphi_-^k \qquad \forall k\in\N.
\]
Consider a periodic point
\[
z_0\in\Fix(\varphi^k)
\]
for some $k\in\N$.
Since the isotopy $\{\varphi_+^{kt}\circ\varphi_-^{kt}\}_{t\in[0,1]}$ is compactly supported and its time-1 map is $\varphi_+^k \circ \varphi_-^k = \varphi^k$, the Conley-Zehnder index $\mu(z_0,\varphi^k)$ is the Conley-Zehnder index of the path of symplectic automorphisms
\begin{equation}\label{E:Z_final}
\Phi:[0,1] \to \mathrm{Sp}(2), \;\; \Phi(t):=D\big(\varphi_+^{kt}\circ\varphi_-^{kt}\big)(z_0) = D\varphi_+^{kt}\big(\varphi_-^{kt}(z_0)\big)\circ D\varphi_-^{kt}(z_0).
\end{equation}
Suppose first that $z_0$ is not in $A$. Then $z_0$ is outside the support of the map $\varphi_-^{k}$ and hence $z_0$ is a fixed point of $\varphi_+^k$ and 
\[
\mu(z_0,\varphi^k) = \mu(z_0,\varphi_+^k).
\]
Therefore, (\ref{E:CZ_Hplus0}) and (\ref{E:CZ_Hplus}) imply
\begin{equation}
\label{CZ_Hfinal1}
-1\leq \mu(z_0,\varphi^k)\leq 2\left\lceil\frac{k}{n} \right \rceil-1 \qquad \forall z_0\in \Fix(\varphi^k) \setminus A.
\end{equation}
In particular, the Conley Zehner index of fixed points of $\varphi^k$ outside of $A$ trivially satisfies the lower bound stated in statement (iii). Moreover, $\overline{\mu}(z_0,\varphi^k) = \overline{\mu}(z_0,\varphi_+^k)$ and, using \eqref{mcz+} and the fact that $\sigma_{\varphi,\lambda}(z_0) = \sigma_{\varphi^+,\lambda}(z_0)$ by \eqref{action_varphi_breakup} and by the fact that $\sigma_{\varphi^-,\lambda}$ vanishes outside $A$, we obtain
\[
\frac{2}{\pi} \sigma_{\varphi,\lambda}(z_0) \leq \frac{\overline{\mu}(z_0,\varphi^k)}{k} \leq \left( \frac{2}{\pi} + \epsilon \right) \sigma_{\varphi,\lambda}(z_0) \qquad \forall z_0\in \Fix(\varphi^k) \setminus A.
\]
Together with (\ref{azfuori}), this bound proves statement (vi).

Now suppose that $z_0$ is in $A$.   We claim that $k=nk_0$ for some $k_0\in\N$ and
that $z_0$ is a fixed point of $\varphi_-^k$.   Indeed, $\varphi_-^k(z_0)$ belongs to $A\subset \D_{\sqrt{1-2\delta}}$ and so 
\[
z_0 = \varphi_+^k\circ\varphi_-^k(z_0) = \rho^k \circ \varphi^k_{-}(z_0).
\]
The point $z_0$ belongs to $\rho^j(A_0)$ for some $j\in \{0,\dots,n-1\}$, and since $\varphi_-$ maps $\rho^j(A_0)$ into itself, $\varphi^k_{-}(z_0)$ belongs to $\rho^j(A_0)$ too. Then the above identity implies that $k$ is a multiple of $n$ and $z_0$ is fixed by $\varphi_-^k$. This proves the claim.  

Moreover, the fact that $z_0$ is in $A$ implies that $\varphi_-^{kt}(z_0)$ is in $A\subset\D_{\sqrt{1-2\delta}}$ for all $t\in[0,1]$, so 
\[
D\varphi_+^{kt}(\varphi_-^{kt}(z_0))=e^{\frac{2\pi}{n} ink_0t}=e^{2\pi i k_0 t} \qquad \forall t\in[0,1],
\]
and (\ref{E:Z_final}) reduces to
\[
\Phi(t)= e^{2\pi ik_0t} \cdot D\varphi_-^{kt}(z_0)\qquad \forall t\in[0,1].
\]
Therefore,
\begin{equation}
\label{relind}
\begin{aligned}
\mu(z_0,\varphi^k) &= \mu(\Phi)=2\, {\rm Maslov}(e^{2\pi ik_0\cdot}) + \mu(D\varphi_-^{k\cdot}(z_0))\\
& = 2k_0 + \mu(z_0,\varphi^k_-) = \frac{2k}{n} + \mu(z_0,\varphi^k_-),
\end{aligned}
\end{equation}
and by \eqref{CZ:Hminus_A} we get
\[
-2nk+\frac{2k}{n}+1 \leq  \mu(z_0,\varphi^k) \leq \frac{2k}{n} -1 \qquad \forall z_0 \in \Fix(
\varphi^k)\cap A.
\]
The above lower bound is precisely the lower bound claimed in statement (iii), whose proof is now complete. 

Now recall that $\theta\in (-n,-n+1)$ has been fixed and set
\[
\nu:= \theta + n.
\]
The number $\nu$ is positive and smaller than $\epsilon$ because of \eqref{extra}. By \eqref{stimaz} we have
\[
\begin{split}
\sigma_{\varphi,\lambda}(z_0) &\geq \frac{\pi}{n} \bigl( (-n+\nu) \eta + 1\bigr) (1-\delta) \geq \frac{\pi}{n} \bigl( (-n+\nu) \eta + 1\bigr) \\ & \geq \frac{\pi}{n} \bigl( -n+\nu + 1) = -\pi + \frac{\pi}{n}(1+\nu).
\end{split}
\]
This proves the first bound
\[
\sigma_{\varphi,\lambda}(z_0) \geq -\pi + \frac{\pi}{n}(1+\nu),
\]
which is stated in (vii).  In order to prove the second bound, notice that the identity \eqref{relind} implies
\[
\overline{\mu}(z_0,\varphi^k) = \frac{2k}{n} + \overline{\mu}(z_0,\varphi^k_-).
\]
Then, the bounds \eqref{mcz-} together with the identity
\[
\sigma_{\varphi_-,\lambda}(z_0) = \sigma_{\varphi,\lambda}(z_0) - \sigma_{\varphi_+,\lambda}(z_0) = \sigma_{\varphi,\lambda}(z_0) - \frac{\pi}{n} (1-\delta),
\]
which follows from \eqref{E:action_Hplus_interior}, give
\[
\begin{split}
\frac{\overline{\mu}(z_0,\varphi^k)}{k} &= \frac{2}{n} + \frac{\overline{\mu}(z_0,\varphi^k_-)}{k} \leq \frac{2}{n} + \frac{2n}{\pi} \sigma_{\varphi^-,\lambda}(z_0) \\ &= \frac{2}{n} + \frac{2n}{\pi} \sigma_{\varphi,\lambda}(z_0) - 2 (1-\delta) = \frac{2}{n} + \frac{2n}{\pi} \sigma_{\varphi,\lambda}(z_0) -2 + 2\delta,
\end{split}
\]
and
\[
\begin{split}
\frac{\overline{\mu}(z_0,\varphi^k)}{k} &= \frac{2}{n} + \frac{\overline{\mu}(z_0,\varphi^k_-)}{k} \geq \frac{2}{n} +\frac{2n}{\eta\pi(1-\delta)}   \sigma_{\varphi_-,\lambda}(z_0) \\ &= \frac{2}{n} + \frac{2n}{\eta\pi(1-\delta)}  \sigma_{\varphi,\lambda}(z_0) -\frac{2n}{\eta\pi(1-\delta)} \frac{\pi}{n} (1-\delta) \\ &= \frac{2}{n} + \frac{2n}{\eta\pi(1-\delta)}  \sigma_{\varphi,\lambda}(z_0) - \frac{2}{\eta}.
\end{split}
\]
The quantity
\[
\frac{2}{n} + \frac{2n}{\eta\pi(1-\delta)}  \sigma - \frac{2}{\eta}
\]
converges to
\[
\frac{2}{n} + \frac{2n}{\pi}  \sigma - 2
\]
for $(\eta,\delta) \rightarrow (1,0)$, uniformly in $\sigma\in [-\pi+\pi/n,\pi/n]$. Therefore, by choosing $\eta$ close enough to 1 and $\delta$ close enough to $0$ we obtain the bounds
\[
\begin{split}
\frac{2}{n} + \frac{2n}{\pi}  \sigma_{\varphi,\lambda}(z_0)  &- 2 - \nu^2 \leq \frac{\overline{\mu}(z_0,\varphi^k)}{k} \\ &\leq \frac{2}{n}  + \frac{2n}{\pi} \sigma_{\varphi,\lambda}(z_0) -2 + \nu^2
\quad \forall z_0 \in \Fix(\varphi^k) \cap A,
\end{split}
\]
concluding the proof of statement (vii). 

Let $w_0\in A_0$ be the point such that $f(w_0)=0$, where $f:A_0 \rightarrow \D_R$ is the area-preserving diffeomorphism introduced at the beginning of the proof. Being fixed by $\varphi_-$, the point $w_0$ is a fixed point of $\varphi^n$.  By \eqref{action_varphi_breakup}, \eqref{E:action_Hplus_interior}, \eqref{phikphimenos} and \eqref{E:action_K}
\[
\begin{split}
\sigma_{\varphi,\lambda}(w_0) &= \sigma_{\varphi_+,\lambda}(w_0) + \sigma_{\varphi_-,\lambda}(w_0) = \frac{\pi}{n} (1-\delta) + \sigma_{\varphi_K,\lambda_0}(0) \\ &= \frac{\pi}{n}(1-\delta) + \eta \frac{\pi}{n} \theta (1-\delta) = \frac{\pi}{n} (1-\delta)(1+\eta \theta) \\ &= \frac{\pi}{n} (1-\delta)(1-\eta n + \eta \nu).
\end{split}
\]
When $\eta$ converges to 1 and $\delta$ converges to 0, the latter quantity converges to
\[
\frac{\pi}{n} (1-n + \nu) = - \pi + \frac{\pi}{n}(1+\nu).
\]
Therefore, by choosing $\eta$ close enough to 1 and $\delta$ small enough we obtain
\[
\sigma_{\varphi,\lambda}(w_0) \leq - \pi + \frac{\pi}{n}(1+\nu) + \nu,
\]
as claimed in statement (viii). By \eqref{cz0}, \eqref{mu=mu} and \eqref{relind} we have for every $h\in \N$
\[
\mu(w_0,\varphi^{hn}) = 2h + \mu(w_0,\varphi^{hn}_-) = 2h + \mu(0,\varphi^{hn}_K) = 2h + 2\lceil hn\theta\rceil - 1.
\]
By dividing by $hn$ and by taking the limit for $h\rightarrow +\infty$ we find
\[
\frac{\overline{\mu}(w_0,\varphi^n)}{n} = \frac{2}{n} + 2\theta = \frac{2}{n} + 2 (-n + \nu) = \frac{2}{n} - 2n + 2\nu,
\]
which is the identity stated in statement (viii). This concludes the proof of Proposition \ref{mapn}

\section{Lifting from the disk to the three-sphere}
\label{secseconda}

\subsection{From the disk to the mapping torus}
\label{SS:disk_to_torus.}

In this section, we will lift a compactly supported area-preserving diffeomorphism of the disk to a Reeb flow on a solid torus. We recall the following statement from~\cite[Proposition 3.1]{abhs17b}, which we give here in a slightly less general form:

\begin{prop}
\label{prop_abhs15} 
Let $\varphi\in \mathrm{Diff}_c(\D,\omega_0)$, let $\lambda$ be a primitive of $\omega_0$ which agrees with $\lambda_0$ outside of a compact subset of the interior of $\D$. Assume that the function
\[
\tau:= \sigma_{\varphi,\lambda} + \pi
\]
is positive. Then there exists a smooth contact form $\beta$ on the solid torus $\D \times \R/\pi \Z$ with the following properties:
\begin{enumerate}[(i)]
\item $\beta = \lambda + ds$ in a neighborhood of $\partial \D \times \R/\pi \Z$, where $s$ denotes the coordinate on $\R/\pi \Z$; in particular, the Reeb vector field $R_{\beta}$ of $\beta$ coincides with $\partial/\partial s$ near the boundary of $\D \times \R/\pi \Z$, and its flow is globally well-defined;
\item for all $s\in \R/\pi \Z$ we have $(\imath_s)^* d\beta = \omega_0$, where $\imath_s: \D \rightarrow \D \times \R/\pi\Z$ denotes the inclusion $z\mapsto (z,s)$;
\item each surface $\D \times \{s\}$ is transverse to the flow of $R_{\beta}$, and the orbit of every point in $\D \times \R/\pi\Z$ intersects $\D \times \{s\}$ both in the future and in the past;
\item $\beta$ is smoothly isotopic to $\lambda + ds$ by a path of contact forms on $\D \times \R/\pi \Z$ which satisfy properties (i), (ii) and (iii) above;
\item the first return map and the first return time of the flow of $R_{\beta}$ associated to the surface $\D\times \{0\} \cong \D$ are the map $\varphi$ and the function $\tau$;
\item $\mathrm{vol}(\D \times \R/\pi \Z, \beta\wedge d\beta) = \pi^2 + \mathrm{CAL}(\varphi)$.
\end{enumerate}
\end{prop}

Let $\varphi$ and $\lambda$ be as in the hypothesis of the above proposition. From this data we get a contact form $\beta$ on $\R/\pi\Z\times \D$ satisfying properties (i)-(vi). By (v), the closed orbits of the Reeb vector field $R_{\beta}$ are in one-to-one correspondence with the periodic points of $\varphi$. More precisely, if $z_0\in \D$ is a fixed point of $\varphi^k$ for some $k\in \N$ then by setting
\[
\gamma(t):= \phi_{R_{\beta}}^t(z_0,0), \qquad T:= \sum_{j=0}^{k-1} \tau(\varphi^j(z_0))
\]
we obtain a closed orbit $(\gamma,T)$ of $R_{\beta}$. Conversely, every closed orbit of  $R_{\beta}$ has this form, up to a time shift. Our aim now is to study the relationship between the Conley-Zehnder index of $z_0$ as a fixed point of $\varphi^k$ and the Conley-Zehnder index of the closed orbit $(\gamma,T)$.

First we need to clarify which trivialization of the contact structure $\ker \beta$ we are going to use to define the Conley-Zehnder index in the Reeb setting. For every $(z,s)\in \D\times \R/\pi\Z$ we denote by 
\[
\Pi_{(z,s)}: \R^3 = T_{(z,s)} (\D \times \R/\pi\Z) \rightarrow \ker \beta(z,s)
\]
the projection along the line spanned by $R_{\beta}(z,s)$. By Proposition \ref{prop_abhs15} (iii), the vector $R_{\beta}(z,s)$ is transverse to $\ker ds = \R^2 \times \{0\}$ and hence the composition
\[
\tilde{\Xi}_{(z,s)} : \R^2 \stackrel{{\rm id}\times 0}{\longrightarrow} \R^2 \times \{0\} = \ker ds \stackrel{\Pi_{(z,s)}}{\longrightarrow} \ker \beta(z,s)
\]
is an isomorphism. The fact that $R_{\beta}$ coincides with $\partial/\partial s$ near the boundary implies that this isomorphism is orientation preserving, and hence
\[
\tilde{\Xi}_{(z,s)}^* d\beta(z,s) = a(z,s)^2 \, \omega_0
\]
for some positive smooth function $a:  \D \times \R/\pi\Z \rightarrow \R$. Then the isomorphism
\begin{equation}
\label{Xibeta}
\Xi_{(z,s)} := a(z,s)^{-1} \,\tilde{\Xi}_{(z,s)} : \R^2 \longrightarrow \ker \beta(z,s)
\end{equation}
satisfies
\[
\Xi_{(z,s)}^* d\beta(z,s) = \omega_0,
\]
and hence $\Xi$ is a global symplectic trivialization of $\ker \beta$. Note that the inverse of this trivialization has the form
\[
\Xi_{(z,s)}^{-1} = a(z,s) \,\tilde{\Xi}_{(z,s)}^{-1} : \ker \beta(z,s) \longrightarrow \R^2,
\]
where $\tilde{\Xi}_{(z,s)}^{-1}$ is given by the composition
\[
\tilde{\Xi}_{(z,s)}^{-1}: \ker \beta(z,s) \stackrel{\hat{\Pi}_{(z,s)}}{\longrightarrow} 
\R^2 \times \{0\} \stackrel{p_1}{\longrightarrow} \R^2,
\]
where
\[
\hat{\Pi}_{(z,s)}: \R^3 = T_{(z,s)} (\D \times \R/\pi\Z) \rightarrow \R^2\times \{0\}
\]
is the projection along the line spanned by $R_{\beta}(z,s)$ and $p_1$ is the projection onto the first factor.

\begin{add}
\label{adde1}
Under the assumptions of Proposition \ref{prop_abhs15}, let $z_0\in \D$ be a fixed point of $\varphi^k$ for some $k\in \N$ and let $(\gamma,T)$ be the corresponding closed orbit of $R_{\beta}$, i.e. 
\[
\gamma(t):= \phi_{R_{\beta}}^t(z_0,0), \qquad T := \sum_{j=0}^{k-1} \tau(\varphi^j(z_0)).
\]
Then the Conley-Zehnder indices of the fixed point $z_0$ and of the closed orbit $(\gamma,T)$ coincide:
\[
\mu(z_0,\varphi^k) = \mu\bigl((\gamma,T),\Xi\bigr).
\]
\end{add}

\begin{proof} 
By Proposition \ref{prop_abhs15} (ii) the restriction of $\beta$ to each $\D\times \{s\}$ differs from the restriction of $\lambda$ by the differential of a function on $\D\times \{s\}$. It follows that
\[
\beta = \lambda + \partial_x f \, dx + \partial_y f\, dy + g\, ds = \lambda + df + h\, ds,
\]
where $f$ and $g$ are smooth real functions on $\D\times \R/\pi\Z$ taking the value 0 and 1, respectively, near the boundary $\partial \D\times \R/\pi\Z$, and $h:= g - \partial_s f$ has the value 1 near the boundary. We consider the following time dependent compactly supported Hamiltonian on the disk:
\[
H_t(z):= h(z,t) - 1,
\]
and consider the nowhere vanishing vector field $Y$ on $\D\times \R/\pi\Z$
\[
Y:= X_{H_s}(z) + \partial_s.
\]
We claim that $Y$ and $R_{\beta}$ are positively collinear. Indeed, this follows from the fact that these vector fields coincide near the boundary and from the following computation
\[
\begin{split}
\imath_Y d\beta =& \imath_{X_{H_s}+\partial_s} \bigl( \omega_0 + dh\wedge ds \bigr) = \imath_{X_{H_s}} \omega_0 + \imath_{ X_{H_s}} (dh\wedge ds) + \imath_{\partial_s} (dh\wedge ds) \\ &= dH_s + dh(X_{H_s}) \, ds + \partial_s h\, ds - dh \\ &= dh - \partial_s h\, ds+ dH_s(X_{H_s}) \, ds + \partial_s h\, ds - dh = 0.
\end{split}
\]
Therefore, the Reeb flow $\phi^t_{R_{\beta}}$ is a positive reparametrization of the flow $\phi^t_Y$ of $Y$:
\begin{equation}
\label{repara}
\phi^t_{R_{\beta}}(\zeta) = \phi^{\eta(t,\zeta)}_Y(\zeta) \qquad \forall (t,\zeta)\in \R \times ( \D \times \R/\pi\Z),
\end{equation}
for a suitable function $\eta=\eta(t,\zeta)$ which is strictly increasing in the first variable.
 Moreover, the form of $Y$ implies that the flow of $Y$ satisfies
\[
\phi_Y^t(z,0) = (\psi^t(z),t)
\]
where $\psi^t$ is the flow of the Hamiltonian vector field $X_{H_t}$ on $\D$. In particular, $\psi^t$ is a path in $\mathrm{Diff_c}(\D,\omega_0)$ such that $\psi^0=\mathrm{id}$. By Proposition \ref{prop_abhs15} (v), $\psi^1=\varphi$ and, more generally, $\psi^k = \varphi^k$ for all $n\in \N$.

Now let $z_0$ be a fixed point of $\varphi^k=\psi^k$ and let $(\gamma,T)$ be the corresponding closed orbit of $R_{\beta}$, as in the statement. In this case,
\begin{equation}
\label{time}
\eta(T,z_0,0) = k.
\end{equation}
The Conley-Zehnder index $\mu((\gamma,T),\Xi)$ is defined as the Conley-Zehnder index of the symplectic path
\[
\Phi: [0,1] \rightarrow \mathrm{Sp}(2), \qquad \Phi(t) := \Xi^{-1}_{\phi^{Tt}_{R_{\beta}}(z_0,0)} \circ d\phi_{R_{\beta}}^{Tt}(z_0,0) \circ \Xi_{(z_0,0)}.
\]
By differentiating (\ref{repara}) we find
\[
d\phi^t_{R_{\beta}}(z_0,0)[u] = d\phi^{\eta(t,z_0,0)}_Y(z_0,0)[u] + d\eta(t,z_0,0)[u] Y\bigl( \phi^{\eta(t,z_0,0)}_Y(z_0,0) \bigr)
\]
for all $u\in T_{(z_0,0)} (\D \times \R/\pi\Z) = \R^3$. Therefore, for every $v\in \R^2$ we have
\[
\begin{split}
\Phi(t) v =  \Xi_{\phi^{Tt}_{R_{\beta}}(z_0,0)}^{-1} \Bigl( &d\phi^{\eta(Tt,z_0,0)}_Y(z_0,0)[\Xi_{(z_0,0)} v] \\ &+ d\eta(Tt,z_0,0)[\Xi_{(z_0,0)} v] Y\bigl( \phi^{\eta(Tt,z_0,0)}_Y(z_0,0) \bigr) \Bigr)
\end{split}
\] 
By using the definition of $\Xi$ we find
\[
\begin{split}
\Phi(t) v =  a(\phi^{Tt}_{R_{\beta}}(z_0,0)) p_1 &\circ \hat{\Pi}_{\phi^{Tt}_{R_{\beta}}(z_0,0)} \Bigl( d\phi^{\eta(Tt,z_0,0)}_Y(z,0)[a(z_0,0)^{-1}\Pi_{(z_0,0)} (v,0)] \\ &+ d\eta(Tt,z_0,0)[\Xi_{(z_0,0)} v] Y\bigl( \phi^{\eta(Tt,z_0,0)}_Y(z_0,0) \bigr) \Bigr)
\end{split}
\] 
Using the fact that $Y$ is in the kernel of $\hat{\Pi}$ we can simplify this formula and get
\[
\Phi(t) v =  a(\phi^{Tt}_{R_{\beta}}(z_0,0)) a(z_0,0)^{-1} p_1\circ \hat{\Pi}_{\phi^{Tt}_{R_{\beta}}(z_0,0)} d\phi^{\eta(Tt,z_0,0)}_Y(z_0,0)[\Pi_{(z_0,0)} (v,0)].
\]
The vector $(\mathrm{id} - \Pi_{(z_0,0)}) (v,0)$ belongs to the line spanned by $Y(z_0,0)$, and hence $d\phi^{\eta(Tt,z_0,0)}(z_0,0)$ maps it into a multiple of $Y(\phi^{\eta(Tt,z_0,0)}(z_0,0))= Y(\phi^{Tt}_{R_{\beta}}(z_0,0))$, which is in the kernel of $\hat{\Pi}_{\phi^{Tt}_{R_{\beta}}(z_0,0)}$. Therefore, we can replace $\Pi_{(z_0,0)} (v,0)$ by $(v,0)$ in the above expression and obtain
\[
\begin{split}
\Phi(t) v &=  a(\phi^{Tt}_{R_{\beta}}(z_0,0)) a(z_0,0)^{-1} p_1\circ \hat{\Pi}_{\phi^{Tt}_{R_{\beta}}(z_0,0)} d\phi^{\eta(Tt,z_0,0)}_Y(z_0,0)[(v,0)] \\&= a(\phi^{Tt}_{R_{\beta}}(z_0,0)) a(z_0,0)^{-1} p_1 \circ \hat{\Pi}_{\phi^{Tt}_{R_{\beta}}(z_0,0)} (d\psi^{\eta(Tt,z_0,0)}(z_0)[v],0) \\ &=  a(\phi^{Tt}_{R_{\beta}}(z_0,0)) a(z_0,0)^{-1} d\psi^{\eta(Tt,z_0,0)}(z)[v],
\end{split}
\]
where we have used the properties of the projector $\hat{\Pi}$. The fact that both $\Phi(t)$ and $d\psi^{\eta(Tt,z_0,0)}(z_0)$ preserve $\omega_0$ implies that the product $a(\phi^{Tt}_{R_{\beta}}(z_0,0)) a(z_0,0)^{-1}$ is constantly equal to 1, and gives us the formula
\[
\Phi(t) v = d\psi^{\eta(Tt,z_0,0)}(z)[v].
\]
By (\ref{time}), the function $t\mapsto \eta(Tt,z_0,0)$ maps the interval $[0,1]$ monotonically increasing onto the interval $[0,k]$. Therefore, the Conley-Zehnder index of $\Phi$ coincides with the Conley-Zehnder index of the path
\[
[0,1] \rightarrow d\psi^{kt}(z_0).
\]
Since $t\mapsto \psi^{kt}$ is a path in $\mathrm{Diff_c}(\D,\omega_0)$ starting at the identity and ending at $\psi^k=\varphi^k$, the Conley-Zehnder index of the above path gives us the Conley-Zehnder index of $(z_0,\varphi^k)$. This concludes the proof.
\end{proof}

\subsection{From the mapping torus to the three-sphere}
\label{Sec:fromTtoS}

The standard contact form $\alpha_0$ on $S^3:= \{z\in \C^2 \mid |z|=1\}$ is the restriction of the 1-form on $\C^2$
\[
\frac{1}{2} \sum_{j=1}^2 ( x_j\, dy_j - y_j\, dx_j),
\]
where $(x_1+iy_1,x_2+iy_2)$ are the standard coordinates in $\C^2$. The corresponding Reeb vector field is 
\[
R_{\alpha_0}(z) = 2i z,
\]
and its orbits are precisely the intersections of $S^3$ with all complex lines. All these orbits are closed and have period $\pi$. We single out the particular orbit
\[
\Gamma := S^3 \cap (\C\times \{0\}) = S^1 \times \{0\},
\]
where $S^1$ denotes the unit circle in $\C$. We also consider the smooth family of closed disks in $S^3$:
\[
\Sigma_{\varsigma} := \{(z_1,z_2)\in S^3 \mid \mbox{either } z_2 = 0 \mbox{ or } z_2 \neq 0 \mbox{ and } \arg z_2=\varsigma\}, \quad \varsigma\in \R/2\pi \Z.
\]
These disks have the same boundary, namely the circle $\Gamma$. Their interiors $\Sigma_{\varsigma}\setminus \Gamma$ define a smooth foliation of $S^3\setminus \Gamma$ by open disks. We single out one of these discs
\[
\Sigma:= \Sigma_0 = \{ (x_1+i y_1 ,x_2+i y_2)\in S^3 \mid x_2 \geq 0, \; y_2=0\},
\]
and we parametrise it by the map
\begin{equation}
\label{themap}
u: \D \rightarrow \Sigma, \qquad u(x,y) := \Bigl(x+iy,\sqrt{1-x^2-y^2}\Bigr).
\end{equation}
This map is a homeomorphism, and its restriction to the interior of $\D$ is a smooth embedding into $S^3$. The following proposition is proved in \cite[Proposition 3.4]{abhs17b}.

\begin{prop}
\label{liftS3}
Let $\varphi\in \mathrm{Diff}_c(\D,\omega_0)$ and let $\lambda$ be a smooth primitive of $\omega_0$ agreeing with $\lambda_0$ outside of a compact subset of the interior of $\D$. Assume that the function
\[
\tau:= \sigma_{\varphi,\lambda} + \pi
\]
is positive. Then there exists a smooth contact form $\alpha$ on $S^3$ with the following properties:
\begin{enumerate}[(i)]
\item $\alpha$ coincides with $\alpha_0$ in a neighbourhood of $\Gamma$ in $S^3$, and in particular $\Gamma$ is a closed orbit of $R_{\alpha}$;
\item for every $\varsigma \in \R/2\pi \Z$, the restrictions of $d\alpha$ and of $d\alpha_0$ to $\Sigma_{\varsigma}$ coincide;
\item the flow of $R_{\alpha}$ is transverse to the interior of each $\Sigma_{\varsigma}$, and the orbit of every point in $S^3\setminus \Gamma$ intersects the interior of $\Sigma_{\varsigma}$ both in the future and in the past;
\item $\alpha$ is smoothly isotopic to $\alpha_0$ through a path of contact forms which satisfy the conditions (i), (ii) and (iii) above;
\item the first return map and the first return time associated to $\Sigma$ are the map $u\circ \varphi \circ u^{-1}$ and the function $\tau\circ u^{-1}$, where $u$ is the map defined in \eqref{themap}.
\item $\mathrm{vol}(S^3, \alpha\wedge d\alpha) = \pi^2 + \mathrm{CAL}(\varphi)$.
\end{enumerate}
\end{prop}

Indeed, this proposition is a direct consequence of Proposition \ref{prop_abhs15}: After lifting the map $\varphi$ to a contact form $\beta$ on the solid torus $\D\times \R/\pi \Z$ by Proposition \ref{prop_abhs15}, one defines $\alpha$ on $S^3\setminus \Gamma$ as
\[
\alpha := (f^{-1})^* \beta,
\]
where 
\[
f:\mathrm{int}(\D) \times \R/\pi \Z  \to S^3\setminus \Gamma
\]
is the smooth diffeomorphism
\begin{equation}\label{map_f}
f(re^{i\theta},s) = \left( re^{i(\theta+2s)} , \sqrt{1-r^2}e^{i2s} \right).
\end{equation}
This diffeomorphism has a continuous extension to $\D \times \R/\pi \Z$, mapping $\partial \D \times \R/\pi \Z$ onto $\Gamma$. The fact that $\beta=\lambda_0 + ds$ near the boundary of the solid torus and the formula
\[
f^*(\alpha_0) = \lambda_0+ds
\]
imply that $\alpha$ smoothly extends to a contact form - again denoted by $\alpha$ - on the whole $S^3$  coinciding with $\alpha_0$ near $\Gamma$, which is then a closed orbit of $R_{\alpha}$. 

The closed orbits of $R_{\alpha}$ in the complement of the circle $\Gamma$ are in one-to-one correspondence with interior periodic points of the map $\varphi$. More precisely, a fixed point $z_0\in \mathrm{int}(\D)$ of $\varphi^k$ determines the closed orbit $(\gamma,T)$ of $R_{\alpha}$ which is defined by
\begin{equation}
\label{corre2}
\gamma(t):= \phi_{R_{\alpha}}^t(u(z_0)), \qquad T := \sum_{j=0}^{k-1} \tau(\varphi^j(z_0)),
\end{equation}
where $u:\D \rightarrow S^3$ is the map defined in (\ref{themap}), and all closed orbits of $R_{\alpha}$ other than $\Gamma$ are obtained in this way, up to a time shift. Now we wish to complement the above proposition with information about the Conley-Zehnder index of these Reeb orbits. 

We recall that the contact structure $\ker \alpha_0$ has the standard global trivialization $\Xi^{\alpha_0}$ which is defined as
\begin{equation}
\label{trivalpha0}
\Xi^{\alpha_0}_{(z_1,z_2)} : \R^2 \to \ker \alpha_0(z_1,z_2), \qquad 
 (x,y) \mapsto x (-\overline{z}_2,\overline{z}_1) + y (-i\overline{z}_2,i\overline{z}_1), 
\end{equation}
for all $(z_1,z_2)\in S^3 \subset \C^2$ and $(x,y)\in \R^2$.
Indeed, $\ker \alpha_0(z_1,z_2)$ is the complex line which is orthogonal to the complex line spanned by $(z_1,z_2)$ with respect to the standard Hermitian product on $\C^2$, and the above trivialization corresponds to the trivialization
\[
(x,y) \mapsto x \, j \cdot \zeta + y \, k \cdot \zeta,
\]
once $\C^2$ is identified with the field of quaternions by setting 
\[\
\zeta:= \re z_1 + (\im z_1)\,  i + (\re z_2)\,  j + (\im z_2)\, k.
\]
The contact form $\alpha$ constructed in Proposition \ref{liftS3} is isotopic to $\alpha_0$ by a path of contact forms, so by Gray theorem the contact structure $\ker \alpha$ is the image of the contact structure $\ker \alpha_0$ by a diffeomorphism of $S^3$ which is isotopic to the identity. By applying this diffeomorphism to $\Xi^{\alpha_0}$ we obtain a global trivialization $\Xi^{\alpha}$ of $\ker \alpha$, which is uniquely defined up to isotopy. The Conley-Zehnder indices of closed orbits of $R_{\alpha}$ refer to this trivialization.

\begin{add}
\label{adde2}
Under the assumptions of Proposition \ref{liftS3}, let $z_0\in \mathrm{int}(\D)$ be a fixed point of $\varphi^k$ for some $k\in \N$ and let $(\gamma,T)$ be the corresponding closed orbit of $R_{\alpha}$ on $S^3$ as in (\ref{corre2}). Then the Conley-Zehnder indices of the fixed point $z_0$ and of the closed orbit $(\gamma,T)$ are related by the identity
\[
\mu\bigl((\gamma,T),\Xi^{\alpha} \bigr) = \mu(z_0,\varphi^k) +4k.
\]
\end{add}  

\begin{proof}
Denote by $\beta$ the contact form on $\D\times \R/\pi \Z$ which is given by Proposition \ref{prop_abhs15} and by $(\tilde{\gamma},T)$ be the closed orbit of $R_{\beta}$ which corresponds to $z_0$, that is,
\[
\tilde{\gamma}(t) := \phi^t_{R_{\beta}}(z_0,0).
\]
Then $\gamma= f\circ \tilde{\gamma}$. In Addendum \ref{adde1} we have proved that
\begin{equation}
\label{m1}
\mu(z_0,\varphi^k) = \mu\bigl((\tilde{\gamma},T),\Xi^{\beta} \bigr),
\end{equation}
where $\Xi^{\beta}$ is the trivialization of $\ker \beta$ which is defined in (\ref{Xibeta}) under the name of $\Xi$. By pulling back the trivialization $\Xi^{\alpha}$ of $\ker \alpha$ by the diffeomorphism $f: \mathrm{int}(\D) \times \R/\pi \Z \rightarrow S^3 \setminus \Gamma$ we obtain another trivialization of $\ker \beta$, which we denote by $f^* \Xi^{\alpha}$, and
\begin{equation}
\label{m2}
\mu\bigl((\gamma,T),\Xi^{\alpha} \bigr) = \mu\bigl((\tilde{\gamma},T),f^*\Xi^{\alpha} \bigr).
\end{equation}
From (\ref{chtriv}) we know that
\begin{equation}
\label{m3}
\mu\bigl((\tilde{\gamma},T),f^*\Xi^{\alpha} \bigr) = \mu\bigl((\tilde{\gamma},T),\Xi^{\beta} \bigr) + 2 \, \mathrm{Maslov} (\Psi),
\end{equation}
where $\Psi: \R/\Z \rightarrow \mathrm{Sp}(2)$ is the loop in the symplectic group
\[
\Psi(t) := (f^*\Xi^{\alpha})_{\tilde{\gamma}(Tt)}^{-1} \circ \Xi^{\beta}_{\tilde{\gamma(Tt)}}.
\]
We claim that the loop $\Psi$ is freely homotopic to the loop
\[
\Psi_1(t) := (f^*\Xi^{\alpha_0})_{\tilde{\gamma}(Tt)}^{-1} \circ \Xi^{\beta_0}_{\tilde{\gamma(Tt)}},
\]
where $\Xi^{\alpha_0}$ is the standard symplectic trivialization of $\ker \alpha_0$, see (\ref{trivalpha0}), and $\Xi^{\beta_0}$ is the symplectic trivialization of $\ker{\beta_0}$ given by (\ref{Xibeta}) in the special case in which $\beta$ is the contact form
\[
\beta_0:=\lambda_0 + ds,
\]
$\lambda_0$ being the rotationally invariant primitive of $\omega_0$ from \eqref{omegalambda}. Indeed, by Proposition \ref{prop_abhs15} (iv), the contact form $\beta$ is smoothly isotopic relative to a neighborhood of the boundary and through contact forms to $\lambda+ds$, and the Reeb vector fields of the contact forms forming this isotopy are transverse to $\ker ds$. Since $\lambda-\lambda_0$ is exact, we can easily prolongate this isotopy to the contact form $\beta_0=\lambda_0+ds$, by keeping all the above properties. Denote by $\{\beta_r\}_{r\in [0,1]}$, the resulting smooth isotopy of contact forms from $\beta_0$ to $\beta$ having the above properties. The fact that $R_{\beta_r}$ is transverse to $\ker ds$ allows us to define a smooth family of symplectic trivializations $\Xi^{\beta_r}$ of $\ker \beta_r$ as in \eqref{Xibeta}. This path of trivializations joins $\Xi^{\beta_0}$ to $\Xi^{\beta}$.

The 1-form $\beta_r$ is the pull-back by $f$ of a uniquely defined smooth 1-form $\alpha_r$ on $S^3$, and $\{\alpha_r\}_{r\in [0,1]}$ is a path of contact forms joining $\alpha_0$ to $\alpha$. The symplectic trivialization $\Xi^{\alpha_0}$ and Gray theorem define a smooth family of symplectic trivializations $\Xi^{\alpha_r}$ of $\ker \alpha_r$ joining $\Xi^{\alpha_0}$ to $\Xi^{\alpha}$. Then 
\[
(r,t) \mapsto (f^*\Xi^{\alpha_r})_{\tilde{\gamma}(Tt)}^{-1} \circ \Xi^{\beta_r}_{\tilde{\gamma(Tt)}}
\]
is the required free homotopy between $\Psi_1$ and $\Psi$. This concludes the proof of the claim.

Since the loop 
\[
\R/\Z \to \mathrm{int}(\D) \times \R/\pi \Z, \qquad t\mapsto \tilde{\gamma}(Tt),
\]
is freely homotopic to the loop
\[
\R/\Z \to \mathrm{int}(\D) \times \R/\pi \Z, \qquad t\mapsto (0,k \pi t),
\]
the loop $\Psi_1$ is freely homotopic to the loop
\[
\Psi_0:\R/\Z \rightarrow \mathrm{Sp}(2), \qquad t\mapsto (f^*\Xi^{\alpha_0})_{(0,k\pi t)}^{-1} \circ \Xi^{\beta_0}_{(0,k\pi t)},
\]
which we are now going to determine. Since 
\[
\ker \beta_0(0,s) = \ker ds = \R^2 \times \{0\},
\]
we deduce from (\ref{Xibeta}) that
\[
\Xi^{\beta_0}_{(0,s)} (x,y) = (x,y), \qquad \forall s\in \R/\pi \Z, \; (x,y)\in \R^2.
\]
The differential of $f$ at $(0,s)$ is 
\[
df(0,s)[(\zeta,\sigma)] = \bigl( \zeta e^{2is}, 2i e^{2is} \sigma), \quad  \forall (\zeta,\sigma)\in T_{(0,s)} (\D \times \R/\pi \Z) = \R^2 \times \R,
\]
and its inverse is
\[
df(0,s)^{-1}[(\zeta_1,\zeta_2)] = \Bigl(e^{-2is} \zeta_1, - \frac{i}{2} e^{-2is} \zeta_2 \Bigr), \quad  \forall (\zeta_1,\zeta_2)\in T_{f(0,s)} S^3 \subset \C^2.
\]
From the above identity and from the expression (\ref{trivalpha0}) for $\Xi^{\alpha_0}$ we find for every $s\in \R/\pi \Z$ and $(x,y)\in \R^2$
\[
\begin{split}
(f^*\Xi^{\alpha_0})_{(0,s)} (x,y) &= df(0,s)^{-1} \circ \Xi^{\alpha_0}_{f(0,s)} (x,y) = df(0,s)^{-1} \circ \Xi^{\alpha_0}_{(0,e^{2is})} (x,y)\\ &= df(0,s)^{-1} \bigl( x (-e^{-2 i s},0) + y (-i e^{-2i s},0) \bigr) \\ &= ( - e^{-4is} (x+iy),0),
\end{split}
\]
and hence
\[
(f^*\Xi^{\alpha_0})_{(0,s)}^{-1} (x+iy,0) = -e^{4is} (x+iy),
\]
where we are identifying the symplectic plane $\R^2$ with $\C$.
Therefore, the loop $\Psi_0: \R/\Z \rightarrow \mathrm{Sp}(2)$ has the form
\[
\Psi_0(t) = - e^{4\pi i k t}.
\]
This loop has Maslov index
\[
\mathrm{Maslov}(\Psi_0) = 2k.
\]
Therefore, also the freely homotopic loop $\Psi$ has Maslov index $2k$, and the identities (\ref{m1}), (\ref{m2}) and (\ref{m3}) imply
\[
\mu\bigl( (\gamma,T), \Xi^{\alpha} \bigr) = \mu(z_0,\varphi^k) + 4k.
\]
\end{proof}

\subsection{Proof of Theorem \ref{T:main_DC}}

Given $\epsilon>0$, let $\lambda$ be the primitive of $\omega_0$ and $\varphi$ the map in ${\rm Diff_c}(\D,\omega_0)$ whose existence is established in Proposition \ref{map2}. Since $\lambda=\lambda_0$ near the boundary of $\D$, we can lift the data $(\lambda,\varphi)$ to $S^3$ and obtain a contact form which satisfies all the requirements of Proposition \ref{liftS3} and Addendum \ref{adde2}. Being isotopic to $\alpha_0$ through contact forms (Proposition \ref{liftS3} (iv)), $\alpha$ is tight. Our aim is to show that the systolic ratio of $\alpha$ belongs to the interval $(2-\epsilon,2)$ and that $\alpha$ is dynamically convex, thus proving Theorem \ref{T:main_DC}.

We claim that $T_{\min}(\alpha)=\pi$. Indeed, the binding orbit $\Gamma$ has period $\pi$, so it is enough to show that all other closed orbits have period not smaller than $\pi$. By Proposition \ref{liftS3} (v), the closed orbits $(\gamma,T)$ of $\alpha$ are in one-to-one correspondence with the periodic points of $\varphi$ and we have
\[
z_0\in \mathrm{Fix}(\varphi^k) \quad \Rightarrow \quad T= \sum_{j=0}^{k-1} \bigl( \pi + \sigma_{\varphi,\lambda}(\varphi^j(z_0))\bigr) = k \pi + \sum_{j=0}^{k-1} \sigma_{\varphi,\lambda}(\varphi^j(z_0)).
\]
If $k=1$, then $z_0$ is a fixed point of $\varphi$ and we have $\sigma_{\varphi,\lambda}(z_0)\geq 0$ by Proposition \ref{map2} (i), so $T\geq \pi$. If $k\geq 2$, the lower bound $\sigma_{\varphi,\lambda}>-\pi/2$, also proved in  Proposition \ref{map2} (i), implies 
\[
T \geq k \pi - k \cdot \frac{\pi}{2} = k\cdot \frac{\pi}{2} \geq \pi.
\]
This shows that $T_{\min}(\alpha)=\pi$. 

By Proposition \ref{map2} (ii), the Calabi invariant of $\varphi$ belongs to the interval $(-\pi^2/2,-\pi^2/2+\epsilon)$. Then Proposition \ref{liftS3} implies the bounds
\[
\frac{\pi^2}{2} < \mathrm{vol}(S^3,\alpha\wedge d\alpha) = \pi^2 + \mathrm{CAL}(\varphi) < \frac{\pi^2}{2}  + \epsilon.
\]
Therefore, the systolic ratio of $\alpha$ has the bounds
\[
2-\epsilon < 2 \left( 1- \frac{2\epsilon}{\pi^2} \right) < \frac{2}{1+\frac{2\epsilon}{\pi^2}} = \frac{\pi^2}{\frac{\pi^2}{2} + \epsilon} < \rho(\alpha) = \frac{\pi^2}{\mathrm{vol}(S^3,\alpha\wedge d\alpha)} < 2,
\]
as we wished to show.

There remains to check that $\alpha$ is dynamically convex, that is, that the Conley-Zehnder index of every closed orbit $(\gamma,T)$ of $R_{\alpha}$ satisfies
\[
\mu\bigl( (\gamma,T), \Xi^{\alpha} \bigr) \geq 3.
\] 
The binding orbit $\Gamma$ has Conley-Zehnder index 3, and its $k$-th iterate has  Conley-Zehnder index $4k-1\geq 3$. Let $(\gamma,T)$ be another closed orbit of $R_{\alpha}$, corresponding to the fixed point $z_0$ of $\varphi^k$. Then Addendum \ref{adde2} gives us
\[
\mu\bigl( (\gamma,T), \Xi^{\alpha} \bigr) = \mu(z_0,\varphi^k) + 4k.
\]
If $k=1$, then $z_0$ is a fixed point of $\varphi$ and $\mu(z_0,\varphi)\geq -1$ by Proposition \ref{map2} (iii), so in this case
\[
\mu\bigl( (\gamma,T), \Xi^{\alpha} \bigr) \geq -1 + 4 = 3.
\]
If $k\geq 2$, then Proposition \ref{map2} (iii) tells us that $\mu(z_0,\varphi^k)\geq 1-3k$ and hence
\[
\mu\bigl( (\gamma,T), \Xi^{\alpha} \bigr) \geq 1-3k + 4k = 1+k\geq 3.
\]
We conclude that $\alpha$ is dynamically convex.

\subsection{Proof of Theorem \ref{thm0}}

Here the real number $\epsilon>0$ and the natural number $n\geq 2$ are given. Fix some small positive number $\epsilon'$. The size of $\epsilon'$ depends on $\epsilon$ and $n$ and will be determined along the way. 
Let $\lambda$ be the primitive of $\omega_0$ and let $\varphi$ be the map in ${\rm Diff_c}(\D,\omega_0)$ which are given by Proposition \ref{mapn} applied to the pair $(n,\epsilon')$. Since
\[
\sigma_{\varphi,\lambda} > -\pi + \frac{\pi}{n} > -\pi,
\]
by Proposition \ref{mapn} (i), Proposition \ref{liftS3} allows us to lift $(\lambda,\varphi)$ to $S^3$ and obtain a contact form $\alpha$ on $S^3$ which satisfies all the requirements listed in Proposition \ref{liftS3} and Addendum \ref{adde2}. Being isotopic to $\alpha_0$ through contact forms (Proposition \ref{liftS3} (iv)), $\alpha$ is tight.

We claim that $T_{\min}(\alpha)=\pi$. The binding orbit $\Gamma$ has period $\pi$, so it is enough to check that all the other closed orbits $(\gamma,T)$ satisfy $T\geq \pi$. Any such orbit corresponds to a fixed point $z_0$ of $\varphi^k$, for some $k\in \N$, and
\begin{equation}
\label{periodo}
T = k \pi + \sum_{j=0}^{k-1} \sigma_{\varphi,\lambda}(\varphi^j(z_0)) = k \bigl( \pi + \sigma_{\varphi,\lambda}(z_0) \bigr),
\end{equation}
where we have used the fact that the action $\sigma_{\varphi,\lambda}$ is $\varphi$-invariant by Proposition \ref{mapn} (iv).
If $z_0$ is a fixed point of $\varphi$, then $\sigma_{\varphi,\lambda}(z_0)\geq 0$ by Proposition \ref{mapn} (i) and hence
\[
T \geq k \pi \geq \pi.
\]
If $z_0$ is not a fixed point of $\varphi$, then $k\geq n$ by Proposition \ref{mapn} (v). Then the lower bound for the action  which is established in Proposition \ref{mapn} (i) implies
\[
T \geq k \left( \pi   - \pi + \frac{\pi}{n} \right) = \frac{k}{n} \pi \geq \pi.
\]
We conclude that $T_{\min}(\alpha)=\pi$.

The bounds on the Calabi invariant of $\varphi$ wich are established in Proposition \ref{mapn} (ii)
and Proposition \ref{liftS3} imply the volume bounds
\[
\frac{\pi^2}{n} < \mathrm{vol}(S^3,\alpha\wedge d\alpha) = \pi^2 + \mathrm{CAL}(\varphi) < \frac{\pi^2}{n}  + \epsilon'.
\]
Therefore, if $\epsilon'$ is small enough then the systolic ratio of $\alpha$ has the bounds
\[
n-\epsilon < \frac{n}{1+\frac{n\epsilon'}{\pi^2}} = \frac{\pi^2}{\frac{\pi^2}{n} + \epsilon'} < \rho(\alpha) = \frac{\pi^2}{\mathrm{vol}(S^3,\alpha\wedge d\alpha)} < n,
\]
as we wished to prove. 

There remains to estimate the mean rotation number of all closed orbits of $R_{\alpha}$. The Conley-Zehnder index of the $k$-th iterate of the binding orbit $\Gamma$ is given by
\[
\mu\bigl( (\Gamma,k\pi), \Xi^{\alpha} \bigr) = 4k-1,
\]
and hence
\begin{equation}
\label{r1}
\overline{\rho}(\Gamma) = \lim_{k\rightarrow +\infty} \frac{\mu\bigl( (\Gamma,k\pi), \Xi^{\alpha} \bigr)}{2k\pi} = \frac{2}{\pi}.
\end{equation}
Let $(\gamma,T)$ be another closed orbit of $R_{\alpha}$. Then $(\gamma,T)$ corresponds to some $z_0\in \Fix(\varphi^k)\cap \mathrm{int}(\D)$ with $k\in \N$, and the period $T$ is given by (\ref{periodo}). By Addendum \ref{adde2} we have for every $h\in \N$
\[
\begin{split}
\frac{\mu\bigl((\gamma,hT),\Xi^{\alpha}\bigr)}{2hT} &= \frac{4hk+\mu(z_0,\varphi^{hk})}{2h \bigl(k\pi + k\sigma_{\varphi,\lambda}(z_0) \bigr)} \\ &= \frac{2}{\pi + \sigma_{\varphi,\lambda}(z_0)} + \frac{1}{2k\bigl( \pi + \sigma_{\varphi,\lambda}(z_0)\bigr)} \frac{\mu(z_0,\varphi^{hk})}{h},
\end{split}
\]
and, by taking a limit for $h\rightarrow +\infty$, we obtain
\begin{equation}
\label{formularho}
\overline{\rho}(\gamma) = \frac{1}{\pi + \sigma_{\varphi,\lambda}(z_0)} \left( 2 + \frac{\overline{\mu}(z_0,\varphi^k)}{2k} \right).
\end{equation}
In the following we set for simplicity
\[
\sigma:= \sigma_{\varphi,\lambda}(z_0).
\]
Assume that the fixed point $z_0$ of $\varphi^k$ does not belong to the $\varphi$-invariant set $A$. Then the upper bound on $\overline{\mu}(z_0,\varphi^k)$ from Proposition \ref{mapn} (vi) implies
\[
\overline{\rho}(\gamma)  \leq \frac{1}{\pi + \sigma} \left( 2 + \frac{\sigma}{\pi} + \frac{\epsilon' \sigma}{2} \right) = \frac{1}{\pi} + \frac{\epsilon'}{2} + \frac{\frac{1}{\pi} - \frac{\epsilon'}{2}}{1+ \frac{\sigma}{\pi}}.
\]
The fact that $z_0$ is not in $A$ implies that $\sigma\geq 0$, see again Proposition \ref{mapn} (vi), and hence we obtain the upper bound
\begin{equation}
\label{r2}
\overline{\rho}(\gamma)  \leq \frac{2}{\pi} \qquad \mbox{if } z_0\notin A.
\end{equation}
Similarly, the lower bound on $\overline{\mu}(z_0,\varphi^k)$ from Proposition \ref{mapn} (vi) gives  us
\[
\overline{\rho}(\gamma)  \geq \frac{1}{\pi+\sigma} \left( 2 + \frac{\sigma}{\pi} \right) = \frac{1}{\pi} + \frac{1}{\pi + \sigma},
\]
and the upper bound $\sigma<\pi/n$ from Proposition \ref{mapn} (i) implies
\begin{equation}
\label{r3}
\overline{\rho}(\gamma) \geq \frac{1}{\pi} \left( 2 - \frac{1}{n+1} \right) \qquad \mbox{if } z_0\notin A.
\end{equation}
Now we assume that $z_0$ belongs to $A$. In this case, the identity \eqref{formularho} and the upper bound on $\overline{\mu}(z_0,\varphi^k)$ from Proposition \ref{mapn} (vii) imply
\[
\overline{\rho}(\gamma)  \leq \frac{1}{\pi + \sigma} \left( 2 + \frac{1}{n} + \frac{n}{\pi} \sigma -1 + \frac{\nu^2}{2} \right) = \frac{n}{\pi} + \frac{1}{\pi+\sigma} \left( -n + 1 + \frac{1}{n} + \frac{\nu^2}{2} \right),
\]
where $\nu$ is a suitable number in the interval $(0,\epsilon')$.
Since $n\geq 2$, choosing $\epsilon'<1$ makes the quantity between parentheses negative, 
and the upper bound on $\sigma$ from Proposition \ref{mapn} (i) gives us
\[
\overline{\rho}(\gamma)  \leq \frac{n}{\pi} + \frac{1}{\pi+\frac{\pi}{n}} \left( -n + 1 + \frac{1}{n} + \frac{\nu^2}{2} \right) = \frac{1}{\pi} \left( 2 - \frac{1}{n+1} + \frac{n}{2(n+1)}\nu^2\right).
\]
By choosing $\epsilon'$ and hence $\nu$ small enough we obtain
\begin{equation}
\label{r4}
\overline{\rho}(\gamma)  < \frac{2}{\pi} \qquad \mbox{if } z_0\in A.
\end{equation}

By the lower bound on $\overline{\mu}(z_0,\varphi^k)$ from Proposition \ref{mapn} (vii) the identity \eqref{formularho} implies
\[
\overline{\rho}(\gamma) \geq \frac{1}{\pi + \sigma} \left( 2 + \frac{1}{n} + \frac{n}{\pi} \sigma -1 - \frac{\nu^2}{2} \right) = \frac{n}{\pi} + \frac{1}{\pi+\sigma} \left( -n + 1 + \frac{1}{n} - \frac{\nu^2}{2} \right).
\]
The fact that the quantity between parenthesis is negative and the lower bound on $\sigma$ from Proposition \ref{mapn} (vii) give us
\[
\begin{split}
\overline{\rho}(\gamma)  &\geq \frac{n}{\pi} + \frac{n}{\pi (1+\nu)} \left( -n + 1 + \frac{1}{n} - \frac{\nu^2}{2} \right) \\ 
&= \frac{-(n-1)^2+2}{\pi(1+\nu)} + \frac{n}{\pi(1+\nu)} \left( \nu - \frac{\nu^2}{2} \right).
\end{split}
\]
When $n\geq 3$ the numerator of the first fraction in the latter expression is negative, so this fraction is larger than $-(n-1)^2+2$ and, by choosing $\epsilon'$ and hence $\nu$ so small that the term $\nu-\nu^2/2$ is positive, we obtain
\[
\overline{\rho}(\gamma) \geq \frac{1}{\pi} \bigl( -(n-1)^2+2  + \delta \bigr),
\]
for some positive number $\delta=\delta(\epsilon')$ which is independent of $z_0\in A$.
When $n=2$ we have
\[
\overline{\rho}(\gamma) \geq \frac{1}{\pi(1+\nu)} + \frac{2}{\pi(1+\nu)} \left( \nu - \frac{\nu^2}{2} \right) = \frac{1}{\pi}  + \frac{1}{\pi(1+\nu)} (\nu-\nu^2),
\]
and, by choosing $\epsilon'$ and hence $\nu$ so small that the term $\nu-\nu^2$ is positive, we obtain
\[
\overline{\rho}(\gamma) \geq \frac{1}{\pi} (1+\delta)= \frac{1}{\pi} \bigl( -(n-1)^2+2 + \delta\bigr)
\]
for some positive number $\delta=\delta(\epsilon')$ which is independent of $z_0\in A$.
Therefore, in either cases we have
\begin{equation}
\label{r5}
\overline{\rho}(\gamma) \geq \frac{1}{\pi} \left( - (n-1)^2 + 2 + \delta \right) \qquad \mbox{if } z_0\in A,
\end{equation}
where the positive number $\delta=\delta(\epsilon')$ is independent of $z_0\in A$.

Finally, we estimate the mean rotation number of the special closed orbit $(\gamma_0,T_0)$ which 
corresponds to the fixed point $w_0\in A$ of $\varphi^n$ whose existence is stated in Proposition \ref{mapn} (viii). By \eqref{formularho} and the expression for the mean Conley-Zehnder index of $w_0$ stated in Proposition \ref{mapn} (viii) we have
\[
\overline{\rho}(\gamma_0) = \frac{1}{\pi + \sigma_{\varphi,\lambda}(w_0)} \left( 2 + \frac{1}{n} - n + \nu \right).
\]
In the case $n\geq 3$, the quantity between parentheses is negative for $\epsilon'$ - and hence $\nu$ - small enugh, so the upper bound on the action stated  Proposition \ref{mapn} (viii) implies
\[
\begin{split}
\overline{\rho}(\gamma_0) &\leq \frac{1}{\pi - \pi + \frac{\pi}{n} (1+\nu)+ \nu} \left( 2 + \frac{1}{n} - n + \nu \right) \\ &= \frac{1}{\pi \left(1+\nu + \frac{n}{\pi} \nu\right)} \left( - (n-1)^2 + 2 + n \nu \right).
\end{split}
\]
By choosing $\epsilon'$ and hence $\nu$ small enough we then obtain the inequality
\begin{equation}
\label{prov}
\overline{\rho}(\gamma_0) \leq\frac{1}{\pi} \left( - (n-1)^2 + 2 + \frac{\epsilon}{2} \right) \qquad \forall n\geq 3.
\end{equation}
In the case $n=2$, the lower bound on the action from Proposition \ref{mapn} (i) gives us
\[
\overline{\rho}(\gamma_0) = \frac{1}{\pi + \sigma_{\varphi,\lambda}(w_0)} \left( \frac{1}{2} + \nu \right) < \frac{1}{\pi - \frac{\pi}{2}} \left( \frac{1}{2} + \nu \right) = \frac{1}{\pi} ( 1 + 2\nu).
\]
By choosing $\epsilon'$ and hence $\nu$ small enough we obtain
\[
\overline{\rho}(\gamma_0) \leq \frac{1}{\pi} \left( 1 + \frac{\epsilon}{2} \right),
\]
and hence the inequality \eqref{prov} holds also in the case $n=2$:
\begin{equation}
\label{r6}
\overline{\rho}(\gamma_0) \leq\frac{1}{\pi} \left( - (n-1)^2 + 2 + \frac{\epsilon}{2} \right) \qquad \forall n\geq 2.
\end{equation}

Now consider the quantities
\[
s(\alpha) = T_{\min}(\alpha) \, \inf_{(\gamma,T) \in \mathcal{P}(\alpha)} \overline{\rho}(\gamma) \qquad \mbox{and} \qquad S(\alpha) = T_{\min}(\alpha) \, \sup_{(\gamma,T) \in \mathcal{P}(\alpha)} \overline{\rho}(\gamma)
\]
defined in the introduction. As we have seen, the contact form $\alpha$ under consideration has $T_{\min}(\alpha)=\pi$ and hence the identity \eqref{r1} and the upper bounds \eqref{r2} and \eqref{r4} imply
\[
S(\alpha) = 2.
\]
The lower bounds \eqref{r3} and \eqref{r5} together with the upper bound \eqref{r6} imply
\[
- (n-1)^2 + 2 < - (n-1)^2 + 2 + \delta \leq  s(\alpha) \leq - (n-1)^2 + 2 + \frac{\epsilon}{2} < - (n-1)^2 + 2 + \epsilon.
\]
Since the fact that $\alpha$ is dynamically convex for $n=2$ has already been checked in the proof of Theorem \ref{T:main_DC}, this concludes the proof of Theorem \ref{thm0}.


\begin{thebibliography}{AAMO08}

\bibitem[ABHS17]{abhs17b}
A.~Abbondandolo, B.~Bramham, U.~L. Hryniewicz, and P.~A.~S. Salom{\~a}o,
  \emph{Sharp systolic inequalities for {R}eeb flows on the three-sphere}, Invent.\ Math., First Online: 16 September 2017.

\bibitem[APB14]{apb14}
J.~C. \'{A}lvarez Paiva and F.~Balacheff, \emph{Contact geometry and
  isosystolic inequalities}, Geom. Funct. Anal. \textbf{24} (2014), 648--669.
  
\bibitem[AAKO14]{ako14}
S.~Artstein-Avidan, R.~Karasev, and Y.~Ostrover, \emph{From symplectic
  measurements to the {M}ahler conjecture}, Duke Math. J. \textbf{163} (2014),
  2003--2022.

\bibitem[AAMO08]{amo08}
S.~Artstein-Avidan, V.~Milman, and Y.~Ostrover, \emph{The {M}-ellipsoid,
  symplectic capacities and volume}, Comm. Math. Helv. \textbf{83} (2008),
  359--369.  

\bibitem[Hof93]{hof93}
H.~Hofer, \emph{Pseudoholomorphic curves in symplectizations with applications
  to the {W}einstein conjecture in dimension three}, Invent. Math. \textbf{114}
  (1993), 515--563.

\bibitem[HWZ98]{hwz98}
H.~Hofer, K.~Wysocki, and E.~Zehnder, \emph{The dynamics on three-dimensional
  strictly convex energy surfaces}, Ann. of Math. \textbf{148} (1998),
  197--289.

\bibitem[HZ90]{hz90}
H.~Hofer and E.~Zehnder, \emph{A new capacity for symplectic manifolds},
  Analysis, et cetera, Academic Press, Boston, MA, 1990, pp.~405--427.

\bibitem[Rab78]{rab78}
P.~H. Rabinowitz, \emph{Periodic solutions of {H}amiltonian systems}, Comm.
  Pure Appl. Math. \textbf{31} (1978), 157--184.

\bibitem[RS95]{rs95}
J.~Robbin and D.~Salamon, \emph{The spectral flow and the {M}aslov index},
  Bull. London Math. Soc. \textbf{27} (1995), 1--33.

\bibitem[Tau07]{tau07}
C.~H. Taubes, \emph{The {S}eiberg-{W}itten equations and the {W}einstein
  conjecture}, Geom. Topol. \textbf{11} (2007), 2117--2202.

\bibitem[Vit00]{vit00}
C.~Viterbo, \emph{Metric and isoperimetric problems in symplectic geometry}, J.
  Amer. Math. Soc. \textbf{13} (2000), 411--431.

\end{thebibliography}

\providecommand{\bysame}{\leavevmode\hbox to3em{\hrulefill}\thinspace}
\providecommand{\MR}{\relax\ifhmode\unskip\space\fi MR }
\providecommand{\MRhref}[2]{%
  \href{http://www.ams.org/mathscinet-getitem?mr=#1}{#2}
}
\providecommand{\href}[2]{#2}

\end{document}